\documentclass{amsart}

\usepackage{amsmath, amsthm, amsfonts, amssymb, amscd, hyperref, marginnote, todonotes, accents, color, multicol}
\usepackage{amsxtra}     
\usepackage[margin=1in]{geometry}
\usepackage{latexsym,xypic,hyperref,rotating}
\usepackage{tikz}
\usetikzlibrary{patterns}
\usetikzlibrary{matrix,arrows}
\usepackage{lscape} 
\usepackage{epsfig}
\usepackage{verbatim}
\usepackage[all,cmtip]{xy}
\usepackage{mathrsfs}
\usepackage{enumerate}
\allowdisplaybreaks
\usepackage[utf8]{inputenc}

\newenvironment{psmallmatrix}
  {\left(\begin{smallmatrix}}
  {\end{smallmatrix}\right)}

\newcommand{\BZ}{{\mathbb Z}}

\newcommand{\sC}{{M}}

\newcommand{\fm}{{\mathfrak m}}
\newcommand{\fn}{{\mathfrak n}}

\newcommand{\sk}{\mathsf k}
\newcommand{\sff}{f}
\newcommand{\sg}{g}
\newcommand{\bsa}{{\boldsymbol a}}
\newcommand{\bsb}{{\boldsymbol b}}

\newcommand{\im}{\operatorname{Im}}

\newcommand{\rank}{\operatorname{rank}}
\newcommand{\HH}[2]{\operatorname{H}_{#1}(#2)}

\newcommand{\sgn}{\operatorname{sgn}}
\newcommand{\Ker}{\operatorname{Ker}}

\newcommand{\depth}[1]{\operatorname{depth}{#1}}

\newcommand{\Hom}[3]{\operatorname{Hom}_{#1}({#2},{#3})}

\newcommand{\edim}{\operatorname{edim}}

\newcommand{\codepth}{\operatorname{codepth}}

\newcommand{\xra}{\xrightarrow}

\newcommand{\Cone}{\operatorname{Cone}}

\theoremstyle{plain}
\newtheorem{theorem}{Theorem}[section]

\newtheorem{lemma}[theorem]{Lemma}

\newtheorem{proposition}[theorem]{Proposition}

\newtheorem*{mainthm}{Theorem A}
\newtheorem*{introcor}{Corollary B}
\theoremstyle{definition}

\newtheorem{remark}[theorem]{Remark}
\newtheorem{definition}[theorem]{Definition}
\newtheorem{construction}[theorem]{Construction}
\newtheorem{setup}[theorem]{Setup}

\newtheorem{example}[theorem]{Example}
\numberwithin{equation}{theorem}
\theoremstyle{remark}

\begin{document}

\title[Iterated Mapping Cones on the Koszul Complex]{Iterated Mapping Cones on the Koszul Complex \\ and Their Application to Complete Intersection Rings}

\author[V.~C.~Nguyen]{Van C.~Nguyen}
\address{Department of Mathematics, United States Naval Academy, Annapolis, MD 21402, U.S.A.}
\email{vnguyen@usna.edu}
\author[O.~Veliche]{Oana Veliche}
\address{Department of Mathematics, Northeastern University, Boston, MA 02115, U.S.A.}
\email{o.veliche@northeastern.edu}

\date{\today} 


\begin{abstract}
Let $(R,\fm,\sk)$ be a complete intersection local ring, $K$ be the Koszul complex on a minimal set of generators of $\fm$, and $A=\HH{}{K}$ be its homology algebra. We establish exact sequences involving direct sums of the components of $A$ and express the images of the maps of these sequences as homologies of iterated mapping cones built on $K$. As an application of this iterated mapping cone construction, we recover a minimal free resolution of the residue field $\sk$ over $R$, independent from the well-known resolution constructed by Tate by adjoining variables and killing cycles. Through our construction, the differential maps can be expressed explicitly as blocks of matrices, arranged in some combinatorial patterns. 
\end{abstract}

\maketitle
\thispagestyle{empty}

\section{Introduction}
\label{sec:Introduction}

Let $(R,\fm, \sk)$ be a local ring with maximal ideal $\fm$, residue field $\sk:= R/\fm$, and codepth $c \geq 1$.  Let $(K,\partial)$ be the Koszul complex on a minimal set of generators of $\fm$  and let $A:=\HH{}{K}= \bigoplus_{0 \leq i \leq c} A_i$ be its homology algebra. In \cite{NV}, the authors analyzed the product structure of the graded-commutative $\sk$-algebra $A$ in lower degrees and used the components of $K$ as building blocks to construct a minimal free resolution $(F,\partial^F)$ of $\sk$ over $R$ up to homological degree five. The higher degrees of this resolution can be extended similarly for some special cases of $R$, but in general it requires an understanding of higher Massey products, which remains elusive. 

For a complete intersection ring $R$, the Koszul homology algebra $A$ is completely determined as the exterior algebra of the first homology $A_1=\HH{1}{K}$, see e.g., \cite[Theorem 6]{T} and \cite[Theorem 2.7]{As}, and a minimal free resolution $(F,\partial^F)$ of $\sk$ was constructed by Tate \cite[Theorem 4]{T}, by adjoining variables of degrees 1 and 2 and killing cycles. In this note, for a complete intersection ring, we show how $A$ can be also used to describe the homology of another construction, namely, the iterated mapping cones $\{M^k\}_{k\geq 0}$ when applied to the Koszul complex $K$, see Construction~\ref{def cone}. We obtain the following main result:

\begin{mainthm}[Proposition \ref{prop:zeta c} and Theorem~\ref{homology c}]
Let $(R,\fm,\sk)$ be a complete intersection ring of codepth $c \geq 1$.   Then,
for any integer $k \geq 0$, there exists an exact sequence on the Koszul homology of the following form:
\begin{equation}
\label{exact sequence}
    0\to A_0^{\oplus{k+c\choose c-1}}\xrightarrow{[\zeta^{k}_1]}A_1^{\oplus{k+c-1\choose c-1}} \xrightarrow{[\zeta^{k-1}_2]}A_2^{\oplus{k+c-2\choose c-1}} \to \cdots \xrightarrow{[\zeta^{k-c+1}_{c}]}A_c^{\oplus{k\choose c-1}} \to  0,
\end{equation}
where the maps $\zeta^k_u$ are given in \eqref{zeta codepth c}. Moreover, the images of $[\zeta^k_u]$ from \eqref{exact sequence} give a complete description for the homology of the iterated mapping cone $\sC^k$, applied to the Koszul complex $K$ as described in Construction~\ref{def cone}, as follows:
\begin{align}
    \HH{}{\sC^{0}} &= A, \ \text{and} \notag\\ 
    \HH{i}{\sC^{k+1}} &\cong \begin{cases}\label{homology cone}
\HH{i}{\sC^{k}}, & \quad 0\leq i \leq 2k \text{ or } i \geq 2k+3+c \\ 
0, & \quad i=2k+1 \text{ or } i=2k+2 \\
\im([\zeta_u^{k+1}]), & \quad \text{for all}\  i=2k+2+u,\ \text{where}\ 1\leq u \leq c.
\end{cases}
\end{align}
\end{mainthm}

The mapping cone can be viewed as a way to ``glue" two complexes together to produce another complex. Here, we apply the iterated mapping cone construction on complexes built from the Koszul complex $K$ (see Construction~\ref{def cone}) and describe the homology of these mapping cones using the maps $[\zeta^k_u]$. Note that our iterated mapping cone construction is different from that being used in the literature to compute resolution of certain (monomial) ideals, see e.g., \cite{CE,HT,S}.

An understanding of the images of the maps $[\zeta^k_u]$ leads us to recover the minimal free resolution $(F,\partial^F)$ of the residue field $\sk$ over any complete intersection ring, and we describe it explicitly using a finite set of data from $K$ and $A_1$, see Section~\ref{subsec:resolution}.  Such a resolution was constructed by Tate \cite[Theorem 4]{T} 
and was well studied by several authors, see e.g., Tate \cite[proof of Theorem 6]{T}, Assmus \cite[Theorem 2.7]{As}, Avramov \cite[\S6]{Av1}, Guliksen and Levin \cite[Proposition 1.5.4]{GL}, Eisenbud \cite[Example on page 42]{E}, Herzog and Martsinkovsky \cite[\S 3]{HM}, and Avramov, Henriques, and Sega \cite[Theorem 2.5]{AHS}. The minimal free resolutions of \emph{arbitrary} finitely generated modules over complete intersections were studied by Eisenbud and Peeva in \cite{EP} by using higher matrix factorizations. See also \cite[\S1]{EP} for an overview of recent developments in minimal free resolutions. 

As a consequence of Theorem A, we obtain the following minimal free resolution of $\sk$. This resolution is derived independently from the Tate construction, as the limit of the mapping cones $\{M^k\}_{k\geq 0}$. We describe the connection between the two constructions in Appendix~\ref{appendix}. 

\begin{introcor}[Theorem~\ref{ci codepth c} and Section~\ref{subsec:DGA}] 
Let $(R,\fm,\sk)$ be a complete intersection ring of codepth $c \geq 1$. Using the notation from \eqref{zeta codepth c}, a minimal free resolution $(F,\partial^F)$ of the residue field $\sk$ over $R$ is given by: 
    \begin{align*}
     F_i&=K_i^{\oplus{ c-1 \choose c-1}} \oplus K^{\oplus{c \choose c-1}}_{i-2}\oplus K^{\oplus{c+1 \choose c-1}}_{i-4}\oplus\cdots \oplus K^{\oplus{c+j-1\choose c-1}}_{i-2j}\oplus\cdots\quad \text{and}\\
 \partial_i^F&=\begin{pmatrix}
 \partial^{\oplus{c-1 \choose c-1}}_i & \zeta^0_{i-1}& 0                &           0       &0&\cdots\\[0.2cm]
           0  & \partial_{i-2}^{\oplus{c \choose c-1}} &\zeta^{1}_{i-3}&           0       &0&\cdots\\[0.2cm]
           0  &                 0 & \partial_{i-4}^{\oplus{c+1 \choose c-1}} &\zeta^2_{i-5}&0&\cdots\\[0.2cm]
      \vdots  &           \vdots  &          \vdots  &           \ddots  &\ddots  &
 \end{pmatrix},
 \end{align*}
for all $i\geq 0$. Moreover, $F$ has a differential graded (DG) algebra structure induced from that of the Koszul complex $(K,\partial)$.
\end{introcor}

Our approach gives a different perspective on the minimal free resolution $(F,\partial^F)$, by highlighting the role of the Koszul homology $A$. First, using \eqref{homology cone} from Theorem A, we show a direct connection between the algebraic structure of $A$ and the construction of $(F,\partial^F)$. Second, our construction is done by fixing certain bases in the involved free components, empowering a block visualization of the differential maps, see Section~\ref{subsec:examples}.

The paper is organized as follows. In Section~\ref{sec:Koszul}, we study the Koszul homology over a complete intersection and obtain exact sequences of Koszul homology in Proposition \ref{prop:zeta c}.
For any commutative ring $R$, in Construction~\ref{def cone}, we describe an iterated mapping cone on a sequence of chain maps between bounded-below complexes of projective $R$-modules. In Section~\ref{subsec:homology}, we apply this iterated mapping cone construction to the Koszul complex of a complete intersection ring and describe the homology of these mapping cones. We show the rows of the induced long exact sequences in Figure~\ref{fig c} split; see Lemma~\ref{imzeta} and Theorem~\ref{homology c}. As an application, in Section~\ref{subsec:resolution} we give a minimal free resolution of $\sk$ over a complete intersection ring and describe its DG-algebra structure in Section~\ref{subsec:DGA}. Finally, in Section~\ref{subsec:examples}, when $R=\sk[x,y,z]_{(x,y,z)}/I$ is a complete intersection ring of embedding dimension three, we provide Examples \ref{exp:codepth2 edim3} and \ref{exp:codepth3 edim3} of codepth 2 and codepth 3, respectively, to illustrate this resolution. We describe the differential maps $\partial_i^F$ as block matrices and observe the patterns for $\partial_{\text{even}}^F$ and $\partial_{\text{odd}}^F$ using  these visual presentations. \\

We provide in Figure~\ref{fig c} the full picture of the long exact sequences of homology induced by \eqref{cone diagram}. By Lemma~\ref{imzeta}, the rows of these long exact sequences split. 
\begin{landscape}
\begin{figure}[h]
 \caption{Long exact sequences of homology}
 \label{fig c}
 \centering
\xymatrixrowsep{3.3pc}
 \xymatrixcolsep{4.1pc}
\small{ \xymatrix
 {
 \vdots \ar@{->}[d] &&&&&&\\
 A_{u-2}^{\oplus{{j+c+1}\choose c-1}} 
 \ar@{->}[d]_{[\zeta_{u-1}^{j+1}]} 
 &&& 
 \vdots\ar@{->}[d]
 &&&\\
 A_{u-1}^{\oplus{{j+c}\choose c-1}} 
 \ar@{->}[d]_{[\zeta_u^{j}]} 
 \ar@{->}[r]^{\HH{2j+u}{\psi^{j+1}}}  &
 \HH{2j+u}{\sC^j} 
 \ar@{->}[r]^{\HH{2j+u}{\sff^j}}
 \ar@{.>}[ld]_{\HH{2j+u}{\sg^{j}}}   & 
 \HH{2j+u}{\sC^{j+1}}
 \ar@{->}[r]^{\HH{2j+u}{\sg^{j+1}}} &
 A_{u-2}^{\oplus{{j+c}\choose c-1}} 
 \ar@{->}[d]_{[\zeta_{u-1}^{j}]} 
 \ar@{->}[r]^{\HH{2j+u-1}{\psi^{j+1}}} \ar@{.>}[ld]_{\HH{2j+u-1}{\psi^{j+1}}} &
 \HH{2j+u-1}{\sC^j}
 \ar@{->}[r] ^{\HH{2j+u-1}{\sff^j}}
 \ar@{.>}[ld]_{\HH{2j+u-1}{\sg^{j}}} &
 \HH{2j+u-1}{\sC^{j+1}} \ar@{->}[r] &
 \vdots \ar@{->}[d] \\
 A_u^{\oplus{{j+c-1}\choose c-1}} 
 \ar@{->}[d] _{[\zeta_{u+1}^{j-1}]} 
 \ar@{->}[r]^{\HH{2j+u-1}{\psi^{j}}} &
 \HH{2j+u-1}{\sC^{j-1}} 
 \ar@{->}[r]^{\HH{2j+u-1}{\sff^{j-1}}} 
 \ar@{.>}[ld]_{\HH{2j+u-1}{\sg^{j-1}}}  & 
 \HH{2j+u-1}{\sC^{j}} 
 \ar@{->}[r]^{\HH{2j+u-1}{\sg^{j}}}  &
 A_{u-1}^{\oplus{{j+c-1}\choose c-1}} 
 \ar@{->}[d]_{[\zeta_u^{j-1}]} 
 \ar@{->}[r]^{\HH{2j+u-2}{\psi^{j}}}
 \ar@{.>}[ld]_{\HH{2j+u-2}{\psi^{j}}}&
 \HH{2j+u-2}{\sC^{j-1}} 
 \ar@{->}[r]^{\HH{2j+u-2}{\sff^{j-1}}}
 \ar@{.>}[ld]_{\HH{2j+u-2}{\sg^{j-1}}} &
 \HH{2j+u-2}{\sC^{j}} 
 \ar@{->}[r]^{\HH{2j+u-2}{\sg^{j}}}  &
 A_{u-2}^{\oplus{{j+c-1}\choose c-1}} 
 \ar@{->}[d]_{[\zeta_{u-1}^{j-1}]} 
 \ar@{.>}[ld]_{\HH{2j+u-3}{\psi^{j}}}  \\
 A_{u+1}^{\oplus{{j+c-2}\choose c-1}} 
 \ar@{->}[d] 
 \ar@{->}[r]^{\HH{2j+u-2}{\psi^{j-1}}} &
 \HH{2j+u-2}{\sC^{j-2}} 
 \ar@{->}[r]^{\HH{2j+u-2}{\sff^{j-2}}}
 \ar@{.>}[ld]_{\HH{2j+u-2}{\sg^{j-2}}} & 
 \HH{2j+u-2}{\sC^{j-1}} 
 \ar@{->}[r]^{\HH{2j+u-2}{\sg^{j-1}}}  &
 A_u^{\oplus{{j+c-2}\choose c-1}} 
 \ar@{->}[d]_{[\zeta_{u+1}^{j-2}]} 
 \ar@{->}[r]^{\HH{2j+u-3}{\psi^{j-1}}}
 \ar@{.>}[ld]_{\HH{2j+u-3}{\psi^{j-1}}}  &
 \HH{2j+u-3}{\sC^{j-2}} 
 \ar@{->}[r]^{\HH{2j+u-3}{\sff^{j-2}}}
 \ar@{.>}[ld]_{\HH{2j+u-3}{\sg^{j-2}}} &
 \HH{2j+u-3}{\sC^{j-1}} 
 \ar@{->}[r] ^{\HH{2j+u-3}{\sg^{j-1}}}  &
 A_{u-1}^{\oplus{{j+c-2}\choose c-1}}
 \ar@{->}[d]_{[\zeta_u^{j-2}]} 
 \ar@{.>}[ld]_{\HH{2j+u-4}{\psi^{j-1}}} \\
 \vdots \ar@{->}[r] &
 \HH{2j+u-3}{\sC^{j-3}} 
 \ar@{->}[r]^{\HH{2j+u-3}{\sff^{j-3}}} & 
 \HH{2j+u-3}{\sC^{j-2}} 
 \ar@{->}[r]^{\HH{2j+u-3}{\sg^{j-2}}}  &
 A_{u+1}^{\oplus{{j+c-3}\choose c-1}} 
 \ar@{->}[d] 
 \ar@{->}[r] ^{\HH{2j+u-4}{\psi^{j-2}}}&
 \HH{2j+u-4}{\sC^{j-3}} 
 \ar@{->}[r]^{\HH{2j+u-4}{\sff^{j-3}}} 
 \ar@{.>}[ld]_{\HH{2j+u-4}{\sg^{j-3}}} 
 &
 \HH{2j+u-4}{\sC^{j-2}} 
 \ar@{->}[r]^{\HH{2j+u-4}{\sg^{j-2}}}   &
 A_u^{\oplus{{j+c-3}\choose c-1}} 
 \ar@{->}[d]_{[\zeta_{u+1}^{j-3}]}  
 \ar@{.>}[ld]_{\HH{2j+u-5}{\psi^{j-2}}} \\
 &&&\vdots \ar@{->}[r] &
 \HH{2j+u-5}{\sC^{j-4}} 
 \ar@{->}[r] ^{\HH{2j+u-5}{\sff^{j-4}}}&
 \HH{2j+u-5}{\sC^{j-3}} 
 \ar@{->}[r]^{\HH{2j+u-5}{\sg^{j-3}}}   &
 A_{u+1}^{\oplus{{j+c-4}\choose c-1}} 
 \ar@{->}[d] \\
 &&&&&&\vdots
 }
 }
\end{figure}
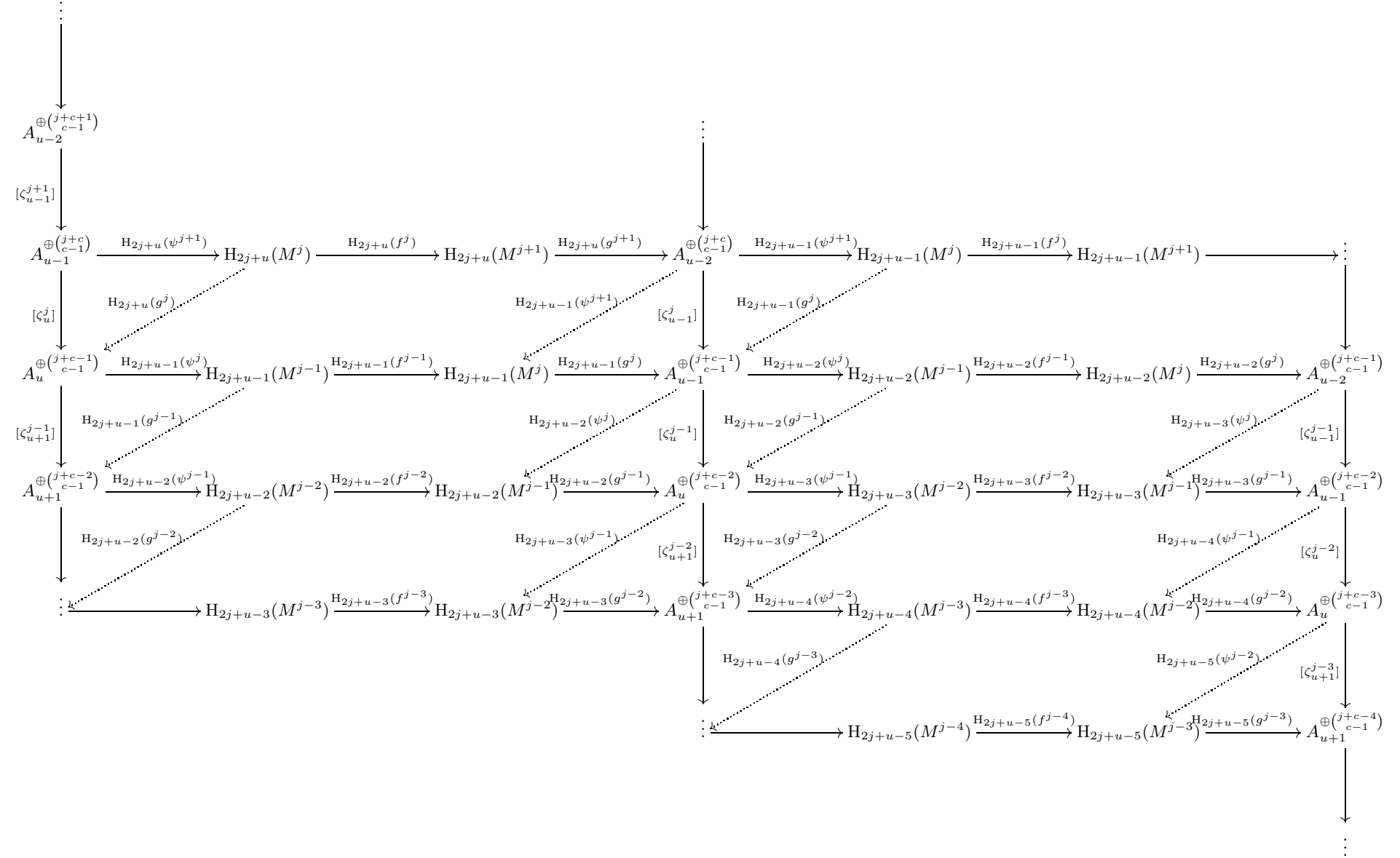
\end{landscape} 

\section{Koszul Homology of a Complete Intersection Ring}
\label{sec:Koszul}

Let $(R,\fm,\sk)$ be a complete intersection local ring of embedding dimension $n$, with maximal ideal $\fm$ and residue field $\sk$.
Assume that $R$ has $\codepth R=\edim R-\depth R=c \geq 1$. Let $K=K^R$ be the Koszul complex of $R$ on a minimal set of generators of $\fm$. The differential graded algebra structure on $K$ induces a graded-commutative algebra structure on the homology $A:=\HH{}{K}= A_0\oplus A_1\oplus\dots\oplus A_c$ of $K$ as a $\sk$-vector space. Since $R$ is a complete intersection, $A$ is the exterior algebra $\bigwedge^\bullet A_1$. In this section, we prove the exactness of sequences \eqref{ses codepth c} on the Koszul homology. 

\medskip

\noindent
{\bf Convention.} The following notation is used for the remaining of the paper.  The differential map on $K$ is denoted by $\partial$. Let $\{z_j\}_{1 \leq j \leq c}$ be elements in $\Ker(\partial_1)$ such that $\{[z_j]\}_{1 \leq j \leq c}$ is a $\sk$-basis for $A_1$. Since $\Ker(\partial_1) \subseteq\fm K_1$, all entries of the vectors $z_j$ are in $\fm.$

\medskip

For any integer $\ell \geq 0$, let $(K^{\oplus \ell},\partial^{\oplus \ell})$ be the direct sum complex with $K^{\oplus \ell} = K \oplus K \oplus \cdots \oplus K$ and $\partial^{\oplus \ell}=\partial \oplus \cdots \oplus \partial$ be a direct sum $\ell$ times.
Here, we use a known combinatorial fact that:
\[ |\{(u_1,\dots,u_k): 1 \leq u_1 \leq \cdots \leq u_k \leq c\}| = {k+c-1 \choose c-1}. \]
Hence, each element $\pi\in K^{\oplus{k+c-1 \choose c-1}}$ has the direct sum indices given by non-decreasingly ordered $k$-tuples as follows:
\[\pi=(\pi_{u_1\dots u_{k}})_{1\leq u_1\leq\dots\leq u_{k}\leq c} \in K^{\oplus{k+c-1 \choose c-1}}. \]
When the subscript is not in non-decreasing order, we use the square bracket notation $[u_1\dots u_k]$ to reorder it non-decreasingly. Furthermore, if $[\pi_{u_1\dots u_k}] \in A_u$ is a homogeneous element of homological degree $u$, then as $\{[z_{1}] \wedge \dots \wedge [z_{s_u}]\}_{1\leq s_1<\dots< s_u\leq c}$ forms a basis in $A_u$, we can uniquely write 
\[ 
[\pi_{u_1\dots u_k}]=
\sum_{1\leq s_1<\dots< s_u\leq c}
a_{u_1\dots u_k}^{s_1\dots s_u}
[z_{s_1}]\wedge \dots\wedge [z_{s_u}],
\]
with coefficients $a_{u_1\dots u_k}^{s_1\dots s_u} \in \sk$. Note that the superscript of $a_{u_1\dots u_k}^{s_1\dots s_u}$ is arranged in \emph{strictly increasing} order, while its subscript can have repeated terms and can be rearranged in \emph{non-decreasing} order by using the square bracket notation $[u_1 \dots u_k]$.
Based on this indexing, we order the components of 
$\pi=(\pi_{u_1\dots u_{k}})_{1\leq u_1\leq\dots\leq u_{k}\leq c}$
using lexicographic order.

\begin{lemma}
\label{zeta chain c}  
Let $(K,\partial)$ be the Koszul complex  on a minimal set of generators of the maximal ideal of a complete intersection local ring  of codepth $c \geq 1$. 
For each integer $k\geq 0$, define the map 
    \begin{equation}   
    \label{zeta codepth c}
    \zeta^{k} \colon \Sigma K ^{\oplus{k+c\choose c-1}} \to  K^{\oplus{k+c-1\choose{c-1}}}
     \quad
    \text{by}
    \quad
      (\pi_{u_1\dots u_{k+1}})_{1\leq u_1\leq\dots\leq u_{k+1}\leq c} \mapsto 
      \Big( \sum_{j=1}^c z_j\wedge \pi_{[v_1\dots v_k j]} \Big)_{1\leq v_1\leq\dots\leq v_k\leq c}.
    \end{equation}
Then, the following assertions hold for all $k\geq 0$:
\begin{enumerate}
 \item[$(a)$] $\zeta^{k}$ is a chain map, and
 \item[$(b)$] $\zeta^{k} \circ \Sigma\zeta^{k+1}= 0$.
\end{enumerate}
\end{lemma}

\begin{proof} (a): The map $\zeta^{k}$ is a chain map if  for all $u \geq 0$  the following diagram is commutative:
  \begin{equation}
  \xymatrixrowsep{3pc}
  \xymatrixcolsep{3pc}
  \xymatrix
  {
  K_{u-1}^{\oplus{k+c\choose c-1}}
  \ar@{->}[d]_{-\partial_{u-1}^{\oplus{k+c\choose c-1}}}
  \ar@{->}[r]^{\zeta_u^k}
  & K_{u}^{\oplus{k+c-1\choose c-1}}
  \ar@{->}[d]^{\partial_u^{\oplus{k+c-1\choose c-1}}}
  \\
  K_{u-2}^{\oplus{k+c\choose c-1}}
  \ar@{->}[r]_{\zeta_{u-1}^k}
  &  K_{u-1}^{\oplus{k+c-1\choose c-1}}.
 }
\end{equation}
Indeed, for $(\pi_{u_1u_2\dots u_{k+1}})_{1\leq u_1\leq\dots\leq u_{k+1}\leq c}\in K_{u-1}^{\oplus{k+c\choose c-1}}$ we have:
\begin{align*}
-\zeta^{k}_{u-1}\circ \partial_{u-1}^{\oplus{k+c\choose c-1}} & \big((\pi_{u_1\dots u_{k+1}})_{1\leq u_1\leq\dots\leq u_{k+1}\leq c} \big) \\ &=-\zeta^{k}_{u-1}
   \big(\partial_{u-1}(\pi_{u_1\dots u_{k+1}}) \big)_{1\leq u_1\leq\dots\leq u_{k+1}\leq c}\\
   &=\Big(-\sum_{j=1}^c z_j\wedge \partial_{u-1}(\pi_{[v_1\dots v_k j]}) \Big)_{1\leq v_1\leq\dots\leq v_k\leq c}\\
   &=\Big(\partial_{u}\big(\sum_{j=1}^c z_j\wedge \pi_{[v_1\dots v_k j]} \big)\Big)_{1\leq v_1\leq\dots\leq v_k\leq c}\\
   &=\partial_u^{\oplus{k+c-1\choose c-1}} \Big( \sum_{j=1}^c z_j\wedge \pi_{[v_1\dots v_k j]} \Big)_{1\leq v_1\leq\dots\leq v_k\leq c}\\
   &=\partial_u^{\oplus{k+c-1\choose c-1}} \circ\zeta_u^{k} \big((\pi_{u_1\dots u_{k+1}})_{1\leq u_1\leq\dots\leq u_{k+1}\leq c} \big),
\end{align*}
where the third equality follows from the Leibniz rule and $\partial_1(z_j)=0$ for all $1\leq j\leq c$. 

(b): Let $(\pi_{w_1w_2\dots w_{k+2}})_{1\leq w_1\leq\dots\leq w_{k+2}\leq c}\in K_{u-2}^{\oplus{k+c+1\choose c-1}}$, then
\begin{align*}
    \zeta^{k} \circ \Sigma\zeta^{k+1}&
    \big( (\pi_{w_1\dots w_{k+2}})_{1\leq w_1\leq\dots\leq w_{k+2}\leq c} \big) \\
    &= \zeta^{k} \Big(-\sum_{j=1}^c z_j \wedge \pi_{[u_1\dots u_{k+1} j]} \Big)_{1\leq u_1\leq\dots\leq u_{k+1}\leq c} \\
    &= - \Big( \sum_{\ell=1}^c z_\ell \wedge \big(\sum_{j=1}^c z_j \wedge \pi_{[v_1\dots v_{k} j\, \ell]} \big) \Big)_{1\leq v_1\leq\dots\leq v_{k}\leq c} \\
    &= (0)_{1\leq v_1\leq\dots\leq v_{k}\leq c},
\end{align*}
where the last equality holds by the Koszul relations $z_\ell \wedge z_\ell=0$ and $z_\ell\wedge z_j=- z_j \wedge z_\ell$ for  $\ell \neq j$. Hence, for all $k\geq 0$, we have $\zeta^k \circ \Sigma\zeta^{k+1}= 0$ as desired.
\end{proof}

\begin{proposition}
\label{prop:zeta c} 
Let $(R,\fm,\sk)$ be a complete intersection local ring of codepth $c \geq 1$. Let $K$ be the Koszul complex of $R$ on a minimal set of generators of $\fm$ and $A=A_0\oplus A_1\oplus \cdots \oplus A_c$ be its homology. Then for each integer $k\geq 0$, the following sequence is exact
\begin{equation}
    \label{ses codepth c}
    0\to A_0^{\oplus{k+c\choose c-1}}\xrightarrow{[\zeta^{k}_1]}A_1^{\oplus{k+c-1\choose c-1}} \xrightarrow{[\zeta^{k-1}_2]}A_2^{\oplus{k+c-2\choose c-1}} \to \cdots \xrightarrow{[\zeta^{k-c+1}_{c}]}A_c^{\oplus{k\choose c-1}} \to  0,
\end{equation}
where the maps $\zeta^{k}$ are defined in \eqref{zeta codepth c}.
\end{proposition}

\begin{proof} 
By Lemma \ref{zeta chain c}, the maps $[\zeta^{j}_u]$ are well defined and the sequence of maps in \eqref{ses codepth c} forms a complex. To show that this sequence is exact, it remains to show $\Ker([\zeta_{u+1}^{k-u}]) \subseteq \im([\zeta_u^{k-u+1}])$, for all $0\leq u\leq c$.

Consider an element $[\pi] \in \Ker([\zeta_{u+1}^{k+u}]) \subseteq A_{u}^{\oplus{k+c-u\choose c-1}} = A_{u}^{\oplus{k-u+1+c-1\choose c-1}}$. For ${1\leq \ell_1\leq\dots\leq \ell_{k-u+1} \leq c}$, we write the $\ell_1\dots\ell_{k-u+1}$ component of $[\pi]$ as
\[ [\pi_{\ell_1\dots\ell_{k-u+1}}] =\sum_{1\leq s_1<\cdots<s_u\leq c} a_{\ell_1\dots\ell_{k-u+1}}^{s_1\dots s_{u}}
[z_{s_1}]\wedge\dots\wedge [z_{s_u}],\]
with coefficients in $\sk$. For $1\leq v_1\leq\dots\leq v_{k-u} \leq c$, the $v_1\dots v_{k-u}$ component of $[\zeta_{u+1}^{k-u}(\pi)]$ is:
\begin{align*}
    0&= \sum_{s=1}^{c}[z_s]\wedge \pi_{[v_1\dots v_{k-u}s]} \\
     &=\sum_{s=1}^{c}[z_s]\wedge \Big(\sum_{1\leq s_1<\cdots<s_u\leq c} a_{[v_1\dots v_{k-u}s]}^{s_1\dots s_{u}}
[z_{s_1}]\wedge\dots\wedge [z_{s_u}]\Big)\\
&=\sum_{1\leq s_1<\cdots<s_u\leq c} \ \, \Big(\sum_{s=1}^{c}\sgn\begin{psmallmatrix} s_1\dots s\dots  s_u\\ss_1\dots s_u\end{psmallmatrix} a_{[v_1\dots v_{k-u}s]}^{s_1\dots s_{u}}\Big)
[z_{s_1}]\wedge\dots \wedge[z_s]\wedge\dots\wedge[z_{s_u}]\\
&=\sum_{1\leq t_1<\cdots<t_{u+1}\leq c}\Big(\sum_{i=1}^{u+1}\sgn\begin{psmallmatrix} t_1\dots t_{u+1}\\t_i(t_1\dots t_{u+1}\setminus t_i)\end{psmallmatrix}a_{[v_1\dots v_{k-u}t_i]}^{t_1\dots t_{u+1}\setminus t_i}\Big)[z_{t_1}]\wedge\dots\wedge[z_{t_{u+1}}]\\
&=\sum_{1\leq t_1<\cdots<t_{u+1}\leq c}\Big(\sum_{i=1}^{u+1}(-1)^{i-1}a_{v_1\dots v_{k-u}t_i]}^{t_1\dots t_{u+1}\setminus t_i}\Big)[z_{t_1}]\wedge\dots\wedge[z_{t_{u+1}}].
\end{align*}
Here, we denote the sign of the permutation 
$\begin{psmallmatrix} s_1\dots s\dots s_u\\ss_1\dots s_u\end{psmallmatrix}$ by 
$\sgn\begin{psmallmatrix} s_1\dots s\dots s_u\\ss_1\dots s_u\end{psmallmatrix}$, and $t_1\dots t_{u+1}\setminus t_i$ means that we remove $t_i$ from the expression $t_1\dots t_{u+1}$.
Since 
$\{[z_{t_1}]\wedge\dots \wedge[z_{t_{u+1}}]\}_{1\leq t_1<\dots<t_{u+1}\leq c}$ is a linearly independent set in $A_{u+1}$, we have:
\begin{equation}
\label{rel a}
\sum_{i=1}^{u+1}(-1)^{i-1}a_{[v_1\dots v_{k-u}t_i]}^{t_1\dots t_{u+1}\setminus t_i}=0,\quad \text{for each }
1 \leq t_1 < \cdots < t_{u+1} \leq c.
\end{equation}

For each $1\leq r_1< \dots < r_{u-1} \leq c$ and $1\leq w_1\leq \dots \leq w_{k-u+2}\leq c$ we set:

\begin{equation}
\label{def a}
    a_{w_1w_2\dots w_{k-u+2}}^{r_1\dots r_{u-1}}
    :=
    \begin{cases} 
      a_{w_2\dots w_{k-u+2}}^{w_1r_1\dots r_{u-1}},&\text{if}\  \  w_1< r_1 \\
    0,& \text{otherwise}. \end{cases}
\end{equation}

Using the above notation with $1\leq s_1<\cdots<s_u\leq c$ and ${1\leq \ell_1\leq\dots\leq \ell_{k-u+1} \leq c}$, we prove:
\begin{equation}
\label{iden a}
\sum_{i=1}^{u} 
(-1)^{i-1}
a_{[\ell_1\dots \ell_{k-u+1}s_i]}^{s_1\dots s_{u}\setminus s_i}=a_{\ell_1\dots\ell_{k-u+1}}^{s_1\dots s_{u}}.
\end{equation}

\begin{proof}[Proof of \eqref{iden a}] Since $1\leq s_1<\cdots<s_u\leq c$, to use definition \eqref{def a}, we consider the following three cases.

\underline{Case 1.}  When $\ell_1<s_1$:
\[
\sum_{i=1}^{u} 
(-1)^{i-1}
a_{[\ell_1\ell_2\dots \ell_{k-u+1}s_i]}^{s_1\dots s_{u}\setminus s_i}=\sum_{i=1}^{u} 
(-1)^{i-1}
a_{[\ell_2\dots \ell_{k-u+1}s_i]}^{\ell_1s_1\dots s_{u}\setminus s_i} =a^{s_1\dots s_u}_{\ell_1\dots\ell_{k-u+1}}.
\]
The first equality holds by using the definition \eqref{def a} since $\ell_1 < s_1  \leq s_i$, for any $1 \leq i \leq u$. The second equality follows from \eqref{rel a}.

\underline{Case 2.}  When $\ell_1=s_1$:
\[
\sum_{i=1}^{u} 
(-1)^{i-1}
a_{[\ell_1\ell_2\dots \ell_{k-u+1}s_i]}^{s_1\dots s_{u}\setminus s_i}
=\sum_{i=1}^{u} 
(-1)^{i-1}
a_{[s_1\ell_2\dots \ell_{k-u+1}s_i]}^{s_1\dots s_{u}\setminus s_i} =a^{s_1\dots s_u}_{\ell_1\dots\ell_{k-u+1}}.
\]
The first equality just replaces $\ell_1 = s_1$ in the subscript. The second equality uses the definition \eqref{def a}, in this case, the only nonzero term occurs when $i=1$.

\underline{Case 3.}  When $\ell_1 > s_1$: Suppose $s_j<\ell_1<s_{j+1}$ for some $1 \leq j \leq u-1$
\begin{align*}
\sum_{i=1}^{u} 
(-1)^{i-1}
a_{[\ell_1\ell_2\dots \ell_{k-u+1}s_i]}^{s_1\dots s_{u}\setminus s_i}
&= a^{s_2\dots s_u}_{s_1\ell_1\dots\ell_{k-u+1}}
+
\sum_{i=2}^{j} 
(-1)^{i-1}
a_{s_i\ell_1\ell_2\dots \ell_{k-u+1}}^{s_1\dots s_{u}\setminus s_i} +
\sum_{i=j+1}^{u} 
(-1)^{i-1}
a_{\ell_1[\ell_2\dots \ell_{k-u+1}s_i]}^{s_1\dots s_{u}\setminus s_i} \\
&= a^{s_1\dots s_u}_{\ell_1\dots\ell_{k-u+1}}.
\end{align*}
The first equality is just to break the sum from $i=1$ to $i=u$ into three smaller sums. In the second equality, we apply the definition \eqref{def a} for each term. Since $\ell_1 > s_i > s_1$ for any $1\leq i \leq j$, the second term is zero since $s_i>s_1$, and the third term is also zero since $\ell_1 > s_1$. The relation \eqref{iden a} now holds.
\end{proof}

Now, we define $([\pi']_{w_1\dots w_{k-u+2}})_{1\leq w_1\leq\cdots\leq w_{k-u+2}\leq c}$ to be an element in $A_{u-1}^{\oplus{k+c-u+1\choose c-1}}$ given by   
\begin{equation*}
[\pi']_{w_1\dots w_{k-u+2}}:=\sum_{1\leq r_1<\cdots<r_{u-1}\leq c} a_{w_1\dots w_{k-u+2}}^{r_1\dots r_{u-1}}
[z_{r_1}]\wedge\dots\wedge [z_{r_{u-1}}],
\end{equation*}
where the coefficients $a_{w_1\dots w_{k-u+2}}^{r_1\dots r_{u-1}}$ are defined as in \eqref{def a}.
Applying $[\zeta_{u}^{k-u+1}]$ on this element, its $\ell_1 \dots \ell_{k-u+1}$ component is: 
\begin{align*}
   &\sum_{r=1}^{c}[z_r]\wedge [\pi']_{[\ell_1\dots\ell_{k-u+1}r]}\\
   &= \sum_{r=1}^{c}[z_r]\wedge \Big(\sum_{1\leq r_1<\cdots<r_{u-1}\leq c}
     a_{[\ell_1\dots \ell_{k-u+1}r]}^{r_1\dots r_{u-1}}
[z_{r_1}]\wedge\dots\wedge [z_{r_{u-1}}]\Big)\\
&=\sum_{1\leq r_1<\cdots<r_{u-1}\leq c} \ \, 
\Big(
\sum_{r=1}^c\sgn
\begin{psmallmatrix} 
r_1\dots r\dots r_{u-1}
\\rr_1\dots r_{u-1}
\end{psmallmatrix}
 a_{[\ell_1\dots \ell_{k-u+1}r]}^{r_1\dots r_{u-1}}
[z_{r_1}]\wedge\dots \wedge[z_r]\wedge\dots\wedge[z_{r_{u-1}}]\Big)\\
&=\sum_{1\leq s_1<\cdots<s_{u}\leq c}
\Big(\sum_{i=1}^{u} 
(-1)^{i-1}
a_{[\ell_1\dots \ell_{k-u+1}s_i]}^{s_1\dots s_{u}\setminus s_i}
\Big)
[z_{s_1}]\wedge\dots\wedge[z_{s_{u}}] \\
& =\sum_{1\leq s_1<\cdots<s_u\leq c} a_{\ell_1\dots\ell_{k-u+1}}^{s_1\dots s_{u}}
[z_{s_1}]\wedge\dots\wedge [z_{s_u}]\\
&=[\pi]_{\ell_1\dots\ell_{k-u+1}},
\end{align*}
which is the $\ell_1\dots\ell_{k-u+1}$ component of $[\pi] \in \Ker([\zeta_{u+1}^{k-u}])$, where the fourth equality follows by \eqref{iden a}. Hence, $\Ker([\zeta_{u+1}^{k-u}]) \subseteq \im([\zeta_{u}^{k-u+1}])$, and the sequence \eqref{ses codepth c} is exact.
\end{proof}

\section{Iterated Mapping Cone Construction on the Koszul Complex}
\label{sec:mapping cone}

In this section, let $R$ be any commutative ring. By applying the mapping cone construction (see e.g., \cite[1.5.1]{W}) recursively on a sequence of $R$-complexes, we construct a sequence of iterated mapping cones. When applying the iterated mapping cone construction on the Koszul complex over a complete intersection ring, we use the algebraic structure of the Koszul homology to describe the homology of these mapping cones in Section~\ref{subsec:homology}.

\subsection{Iterated Mapping Cone Construction}

For any complex $C =(C_i,\partial_i^{C})_{i\in\BZ}$, denote by $\Sigma C$ the complex with $(\Sigma C)_i=C_{i-1}$ and $\partial_i^{\Sigma C}=- \partial_{i-1}^{C}$. Recall that for any complex homomorphism $\psi: (C,\partial^C) \to (D,\partial^D)$, the \emph{mapping cone of $\psi$}, denoted by $(\Cone(\psi),\partial)$, is defined to be the complex with $\Cone(\psi)_i = C_{i-1} \oplus D_i$ for all $i$, and differential map $\partial_i: \Cone(\psi)_i \to \Cone(\psi)_{i-1}$ with $\partial_i(c,d) = (-\partial_{i-1}^C(c), \psi_{i-1}(c) + \partial_i^D(d))$.

\begin{construction}
  \label{def cone}
  Let $\{C^{j}\}_{j\geq 0}$ be a sequence of bounded-below complexes of projective $R$-modules and  \[\{\varphi^j\colon C^j\to C^{j-1} \}_{j\geq 1}\]
  be a sequence of chain maps. Set $\sC^0 \colon=C^0$, map $\psi^1\colon =\varphi^1$, and $\sC^1\colon=\Cone(\psi^1)$. Since $C^2$ consists of projective modules, by \cite[Proposition 5.2.2]{CFH} the functor $\Hom{R}{C^2}{-}$ is exact.
  Therefore, by \cite[1.5.2]{W}, there exists a homomorphism $\psi^2: \Sigma C^2\to\sC^1$ such that the following diagram of complexes is commutative:
\[
\xymatrix{
&&&\Sigma C^2\ar@{->}[d]^{\Sigma\varphi^2}\ar@{.>}[dl]_{\psi^2}&\\
0\ar@{->}[r]&\sC^0\ar@{->}[r]^{\sff^0}&\sC^1\ar@{->}[r]^{\sg^1}&\Sigma C^1\ar@{->}[r]&0,
}
\]
where the bottom row is an exact sequence of complexes induced by the mapping cone $\sC^1$. Define  $\sC^2\colon=\Cone(\psi^2)$. 
Inductively, for each integer $j\geq 1$, there exists a chain complex 
\[\psi^{j+1}\colon \Sigma^{j}C^{j+1} \to \sC^j\] 
such that the following diagram of complexes is commutative, with bottom row being exact:
\begin{equation}
  \label{cone diagram}
\xymatrix{
&&&\Sigma^{j} C^{j+1}\ar@{->}[d]^{\Sigma^{j}\varphi^{j+1}}\ar@{.>}[dl]_{\psi^{j+1}}&\\
0\ar@{->}[r]&\sC^{j-1}\ar@{->}[r]^{\sff^{j-1}}&\sC^{j}\ar@{->}[r]^{\sg^{j}}&\Sigma^{j} C^{j}\ar@{->}[r]&0.
}
\end{equation}
Define  
\[\sC^{j+1}\colon=\Cone(\psi^{j+1}).\]
Remark that for each $j\geq 1$, each homological degree $i$ component of the complex $\sC^j$ can be expressed as: 
\begin{equation}
\label{Cj}
    \sC^j_i = C^j_{i-j} \oplus \sC^{j-1}_{i}.
\end{equation}
In particular, each $\sC^j$ is a complex of projective modules. In this way, we obtain an inductive sequence $\{\sC^j,\sff^j\}_{j\geq 0}$ of complexes and injective chain maps. The direct limit of this sequence is:
\begin{equation}
  \label{eq cone}
  \sC \colon=\lim_{j\to\infty}\sC^j.
\end{equation}
By \cite[Example 3.2.9]{CFH}, $\sC$ is in fact the union complex $\bigcup_{j\geq 0}\sC^j$. 
\end{construction}

\begin{definition}
Let $\{\varphi^j: C^j\to C^{j-1} \}_{j\geq 1}$ be a sequence of chain maps of bounded below complexes. The sequence $\{\sC^j,\sff^j\}_{j\geq 0}$ of complexes in Construction \ref{def cone} is an {\it iterated mapping cone sequence of $\{\varphi^j\}_{j\geq 1}$} and the complex $\sC$ in \eqref{eq cone} is a {\it limit mapping cone of $\{\varphi^j\}_{j\geq 1}$}.
\end{definition}

Remark that by construction, $\sC$ depends on the choice of the sequence of lifting maps $\{\psi^j\}_{j\geq 1}$. In the case when $(C^j,\varphi^j)_{j\geq 1}$ is a complex of complexes, we describe these lifting maps $\{\psi^j\}_{j\geq 1}$ explicitly in the following result.

\begin{proposition}
Let $\{\varphi^j: C^j\to C^{j-1} \}_{j\geq 1}$ be a sequence of chain maps of bounded below complexes of projective $R$-modules such that $\varphi^j\circ\varphi^{j+1}=0$ for all $j\geq 1$, that is, 
\[\dots \to C^j\xra{\varphi^j} C^{j-1}\to \dots \to C^2\xra{\varphi^2} C^1\xra{\varphi^1}C^0\] is a complex of complexes. 
Then, for each $j\geq 1$, the map $\psi^{j+1}: \Sigma^{j}C^{j+1}\to\sC^{j}$ in diagram \eqref{cone diagram} can be chosen in each degree $i$ to be the composition 
\[\psi^{j+1}_i: C^{j+1}_{i-j}\xra{\varphi^{j+1}_{i-j}} C^{j}_{i-j}\xra{(-1)^j\iota^j_{i-j}} \sC^{j}_{i}= C^j_{i-j}\oplus\sC^{j-1}_{i}\] where $\iota^j_{i-j}$ is the natural injection of $R$-modules.
\end{proposition}

\begin{proof} Using the expression for $\psi^{j+1}_i$ given in the statement, it is clear that the triangle in diagram \eqref{cone diagram} is commutative, that is, for all $x\in C^{j+1}_{i-j}$ we have 
\begin{equation*}
    \sg^j_i(\psi^{j+1}_i(x))=(-1)^j\varphi^{j+1}_{i-j}(x) = (\Sigma^j \varphi^{j+1})_i(x).
\end{equation*}
It remains to check that $\psi^{j+1}$ is a  chain homomorphism for all $j\geq 1$, that is, we show for each degree $i$:
\begin{equation*}
\partial^{\sC^{j}}_{i}\circ \psi^{j+1}_{i}=(-1)^j\psi^{j+1}_{i-1}\circ \partial^{C^{j+1}}_{i-j}. 
\end{equation*}
For all $x \in C^{j+1}_{i-j}$, the left-hand side of the above equation is
\begin{align*}
    \partial^{\sC^{j}}_{i}\circ \psi^{j+1}_i(x)
    &= (-1)^{j}\partial^{\sC^{j}}_{i}\circ\iota_{i-j}^j\circ\varphi^{j+1}_{i-j}(x)\\
    &= (-1)^{j}\partial^{\sC^{j}}_{i}(\varphi^{j+1}_{i-j}(x),0)\\
    &= (-1)^j\big((-1)^{j}\partial^{C^{j}}_{i-j}(\varphi^{j+1}_{i-j}(x)), -\psi^j_{i-1}(\varphi^{j+1}_{i-j}(x))\big)\\
    &= (-1)^j\big((-1)^{j}\varphi^{j+1}_{i-j-1}(\partial^{C^{j+1}}_{i-j}(x)), -\psi^j_{i-1}(\varphi^{j+1}_{i-j}(x))\big)\\
    &= (-1)^j\big((-1)^{j}\varphi^{j+1}_{i-j-1}(\partial^{C^{j+1}}_{i-j}(x)), 0\big)\\
    &= (-1)^j\psi^{j+1}_{i-1}\big(\partial^{C^{j+1}}_{i-j}(x)\big),
\end{align*}
which is the right-hand side of the desired equation. Here, the third equality applies the differential map on the cone complex $\sC^j$, the fourth equality holds since $\{\varphi^j\}_{j\geq 1}$ is a sequence of chain maps, the fifth equality holds by the assumption $\varphi^j\circ\varphi^{j+1}=0$ and the given expression for $\psi^j$, and the last equality holds by the given expression for $\psi^{j+1}$.
\end{proof}

\begin{remark}
 \label{les isom} 
 Let $\{C^{j}\}_{j\geq 0}$ be a sequence of complexes of projective $R$-modules such that $C^j_{i<0}=0.$
 For each integer $j\geq 1$, by \cite[Theorem 1.3.1]{W}, the short exact sequence in \eqref{cone diagram} induces a long exact sequence of homology
\begin{equation*}
\hspace{0.5cm}
    \begin{tikzpicture}[descr/.style={fill=white, inner sep=1.5pt}]
      \matrix(m)[
      matrix of math nodes,
      row sep=3 em,
      column sep=5em,
      text height=1.5ex, text depth=0.25ex
      ]
      { \cdots\HH{i+1}{\sC^{j-1}} & \HH{i+1}{\sC^{j}}&\HH{i}{\Sigma^{j-1}C^{j}}\\
        \HH{i}{\sC^{j-1}} & \HH{i}{\sC^{j}}&\HH{i-1}{\Sigma^{j-1}C^{j}}\\
        \HH{1}{\sC^{j-1}} & \HH{1}{\sC^{j}}&\HH{0}{\Sigma^{j-1}C^{j}}\\
          \HH{0}{\sC^{j-1}} & \HH{0}{\sC^{j}}&\HH{-1}{\Sigma^{j-1}C^{j}}.\\
      };
      \path[overlay,->, font=\scriptsize,>=latex]
        (m-1-1) edge node[descr,yshift= 1.3ex]{$\HH{i+1}{\sff^{j-1}}$} (m-1-2)
        (m-1-2) edge node[descr,yshift= 1.3ex]{$\HH{i+1}{\sg^{j}}$}(m-1-3)
        (m-1-3) edge[out=355,in=175] node[descr,yshift=0.3ex] {$\HH{i}{\psi^{j}}$} (m-2-1)
        (m-2-1) edge node[descr,yshift= 1.3ex]{$\HH{i}{\sff^{j-1}}$} (m-2-2)
        (m-2-2) edge node[descr,yshift= 1.3ex]{$\HH{i}{\sg^{j}}$}(m-2-3)
        (m-2-3) edge[out=355,in=175,dashed] node[descr,yshift=0.3ex] {} (m-3-1)
        (m-3-1) edge node[descr,yshift= 1.3ex]{$\HH{1}{\sff^{j-1}}$} (m-3-2)
        (m-3-2) edge node[descr,yshift= 1.3ex]{$\HH{1}{\sg^{j}}$}(m-3-3)
        (m-3-3) edge[out=355,in=175] node[descr,yshift=0.3ex] {$\HH{0}{\psi^{j}}$} (m-4-1)
        (m-4-1) edge node[descr,yshift= 1.3ex]{$\HH{0}{\sff^{j-1}}$} (m-4-2)
        (m-4-2) edge node[descr,yshift= 1.3ex]{}(m-4-3);
        \end{tikzpicture}
\end{equation*}
Since $C_i^j=0$ for all $i<0$ and all $j\geq 0$, we have $\HH{i}{\Sigma^{j-1}C^j}\cong\HH{i-j+1}{C^j}=0, \text{ for all } i\leq j-2.$
Therefore, for each $j \geq 1$, the long exact sequence above induces the following consequences:
    \begin{align*} 
    \HH{i}{\sff^{j-1}} &: \HH{i}{\sC^{j-1}}\to\HH{i}{\sC^j} \text{ is an isomorphism for all } i\leq j-2,\ \text{and}\\
\HH{j-1}{\sff^{j-1}} &: \HH{j-1}{\sC^{j-1}} \to\HH{j-1}{\sC^j} \text{ is an monomorphism.} 
    \end{align*}
\end{remark}

\subsection{Homology of the Iterated Mapping Cone on the Koszul Complex}  
\label{subsec:homology}

Let $(R,\fm,\sk)$ be a complete intersection ring of codepth $c \geq 1$. In this subsection, we apply Construction~\ref{def cone} to a sequence of complexes built from the Koszul complex $K$. In Theorem~\ref{homology c}, we provide a description of the homology of these mapping cones.

\begin{setup}
\label{setup codepth c}
Let $(K,\partial)$ be the Koszul complex  on a minimal set of generators of the maximal ideal $\fm$, set
\begin{equation*}
C^0:=K, \quad   
C^j:= \Sigma ^{j}K^{\oplus{j+c-1\choose c-1}},\quad \text{and} \quad \varphi^{j}:=\Sigma^{j-1}\zeta^{j}:C^{j}\to C^{j-1},\quad \text{for any integer} \quad j\geq 1.
\end{equation*}
By Lemma \ref{zeta chain c}, we obtain that $\varphi^{j}$ is a chain map for all $j\geq 1$ and $\{C^j,\varphi^j\}_{j\geq 0}$ is a complex of complexes. We define an iterated mapping cone $\sC=\lim_{j\to\infty}\sC^j$, as in Construction \ref{def cone}.  Observe that $\sC^0= K$, and thus we have the following Koszul homology on a complete intersection ring of codepth $c \geq 1$:
\begin{equation} 
\label{Koszul homology c}
\HH{i}{\sC^0}=
\begin{cases}A_i,&\text{if}\ 0\leq i\leq c\\   
             0,&\text{otherwise}.
\end{cases} 
\end{equation}
For each integer $j\geq 1$, the diagram \eqref{cone diagram} becomes the following diagram with exact rows, where the top row is a shift of the bottom row with a change of index $j$:
\begin{equation}
\label{cone diagram c}
\xymatrixrowsep{2pc}
 \xymatrixcolsep{3pc}
\xymatrix{
0\ar@{->}[r]&
\Sigma\sC^{j}\ar@{->}[r]^{\Sigma\sff^{j}}&
\Sigma\sC^{j+1}\ar@{->}[r]^{\Sigma\sg^{j+1}}&
\Sigma^{2j+1} K^{\oplus{j+c\choose c-1}}\ar@{->}[d]^{\Sigma^{2j}\zeta^{j}}\ar@{.>}[dl]_{\psi^{j+1}}\ar@{->}[r]&0\\
0\ar@{->}[r]&\sC^{j-1}\ar@{->}[r]^{\sff^{j-1}}&\sC^{j}\ar@{->}[r]^{\sg^{j}}&\Sigma^{2j} K^{\oplus{j+c-1\choose c-1}}\ar@{->}[r]&0.
}
\end{equation}
The exact rows induce the rows of long exact sequences of homology in Figure~\ref{fig c}. We note that in Figure~\ref{fig c}, the columns of Koszul homology are indeed the exact sequences \eqref{ses codepth c} in Proposition~\ref{prop:zeta c} and all triangles right and left to the columns in Figure~\ref{fig c} are commutative. For simplicity, we omit the $\pm$ signs on the maps due to shifting.
\end{setup}

\begin{lemma}
\label{imzeta}
Assume the above Setup \ref{setup codepth c}. For all integers $k \geq 0$ and $u\in\BZ$, we have \[\HH{u}{\sff^{k}}=0\quad\text{and}\quad \HH{2k+u}{\sC^k} \cong \im([\zeta_u^{k}]),\] 
where $[\zeta^{k}]$ are maps from the exact sequences \eqref{ses codepth c} in Proposition~\ref{prop:zeta c}. 

In particular, the rows of the long exact sequences in Figure~\ref{fig c} split.
\end{lemma}

\begin{proof}
We prove the lemma by induction on $k \geq 0$. This can be done by diagram chasing on Figure~\ref{fig c} and we give full details here for completeness. 

For the base case $k=0$, we want to show $\HH{u}{\sff^{0}}=0$ for any $u\in\BZ$. The second row of Figure~\ref{fig c}, with $j=1$, yields the following commutative diagram with exact row and exact column:
\[
 \xymatrixcolsep{3pc}
 \xymatrix{  
 &
 &A_{u-2}^{\oplus{c+1\choose c-1}}\ar@{->}[d]_{[\zeta^{1}_{u-1}]}
 \ar@{->}[dl]_{\HH{u+1}{\psi^2}}
 &
 &
 \\
\HH{u+1}{\sC^0} \ar@{->}[r]^{\HH{u+1}{\sff^0}}
 &\HH{u+1}{\sC^1} \ar@{->}[r]^{\HH{u+1}{\sg^1}}
 & A_{u-1}^{\oplus{c\choose c-1}}\ar@{->>}[d]_{[\zeta^{0}_{u}]}\ar@{->}[r]^{\HH{u}{\psi^1}} 
 & \HH{u}{\sC^0}\ar@{->}[r]^{\HH{u}{\sff^0}} \ar@{->}[dl]^{\HH{u}{\sg^0}}_{\cong}
 & \HH{u}{\sC^1}
 \\
 &
 &A_{u}^{\oplus{c-1 \choose c-1}} \ar@{->}[d]
 &
 &
 \\
 &&
 0. &&
 }
\]
By commutativity of the bottom right triangle, we have $[\zeta^{0}_u] = \HH{u}{\sg^0} \circ \HH{u}{\psi^1}$. Since $\HH{u}{\sg^0}$ is an isomorphism and $[\zeta^{0}_u]$ is a surjection, this implies $\HH{u}{\psi^1}$ is a surjection. Hence, $\HH{u}{\sff^0}=0$ and the above row of long exact sequence splits for any $u\in\BZ$. This proves the claim for $k=0$. Moreover, for any $u\in\BZ$, we obtain the following diagram:

\[
 \xymatrixcolsep{3pc}
 \xymatrix{ 
 & 
 A_{u-2}^{\oplus{c+1\choose c-1}}
 \ar@{->}[d]^{}_{\HH{u+1}{\psi^2}} 
 \ar@{->}[r]^{[\zeta_{u-1}^{1}]}
 & A_{u-1}^{\oplus{c \choose c-1}} 
 \ar@{=}[d] 
 \ar@{->}[r]^{[\zeta_u^{0}]}
 &A_u^{\oplus{c-1\choose c-1}} 
 \ar@{->}[d]_{\cong}^{\HH{u}{\sg^0}^{-1}} 
 \ar@{->}[r] 
 & 0 
 \\
 0 
 \ar@{->}[r] 
 & 
 \HH{u+1}{\sC^1} 
 \ar@{->}[r]_{\HH{u+1}{\sg^1}} 
 & A_{u-1}^{\oplus{c\choose c-1}} 
 \ar@{->}[r]_{\HH{u}{\psi^1}} 
 & 
 \HH{u}{\sC^0} 
 \ar@{->}[r]
 & 0. 
 \\
 }
\]
By Snake Lemma $\HH{u+1}{\psi^2}$ is surjective and one can show that $\Ker(\HH{u+1}{\psi^2})=\Ker([\zeta_{u-1}^1])$. Therefore, $\HH{u+1}{\sC^1} \cong \im([\zeta_{u-1}^{1}])$ for any $u\in\BZ$. 
By exactness of the first row in Figure~\ref{fig c}, with $j=1$, we get $\HH{u+1}{\sff^1}=0$ for all $u\in\BZ$, hence the entire long exact sequence splits. This proves the claim for $k=1$.

By induction, for $k \geq 1$, assume $\HH{u}{\sff^{k}}=0$ and $\HH{2k+u-1}{\sC^k} \cong \im([\zeta_{u-1}^{k}])$, for  all $u \in \BZ$. We want to show $\HH{u}{\sff^{k+1}}=0$ and $\HH{2k+u}{\sC^{k+1}} \cong \im([\zeta_{u-2}^{k+1}])$. Figure~\ref{fig c} with $j=k$ yields the following commutative diagram with exact rows: 
\[
\xymatrixrowsep{2.9pc}
\xymatrixcolsep{4.2pc}
\xymatrix{
 \ar@{->}[r]^-{\HH{2k+u+1}{\sff^{k+1}}}
 & \HH{2k+u+1}{\sC^{k+2}}  \ar@{->}[r]^-{\HH{2k+u+1}{\sg^{k+2}}}
 & A_{u-3}^{\oplus{k+c+1 \choose c-1}}
 \ar@{->}[d]_-{[\zeta^{k+1}_{u-2}]}
 \ar@{->}[r]^-{\HH{2k+u}{\psi^{k+2}}} \ar@{->}[dl]_-{\HH{2k+u}{\psi^{k+2}}} 
 & \HH{2k+u}{\sC^{k+1}}
 \ar@{->}[r]^-{\HH{2k+u}{\sff^{k+1}}} \ar@{->}[dl]^-{\HH{2k+u}{\sg^{k+1}}} &
 \\
 \ar@{->}[r]^-{0} 
 & \HH{2k+u}{\sC^{k+1}}
 \ar@{^{(}->}[r]^-{\HH{2k+u}{\sg^{k+1}}}
 & A_{u-2}^{\oplus{k+c \choose c-1}}
 \ar@{->}[d]_-{[\zeta^{k}_{u-1}]}
 \ar@{->>}[r]^-{\HH{2k+u-1}{\psi^{k+1}}} 
 \ar@{->>}[dl]_-{\HH{2k+u-1}{\psi^{k+1}}}
 & \HH{2k+u-1}{\sC^{k}}
 \ar@{->}[r]^-{\HH{2k+u-1}{\sff^k}=0} \ar@{_{(}->}[dl]^-{\HH{2k+u-1}{\sg^{k}}}
 &
 \\
 \ar@{->}[r]^-{0}
 & \HH{2k+u-1}{\sC^{k}}  
 \ar@{^{(}->}[r]^-{\HH{2k+u-1}{\sg^{k}}}
 & A_{u-1}^{\oplus{k+c-1 \choose c-1}}
 \ar@{->>}[r]^-{\HH{2k+u-2}{\psi^{k}}} 
 & \HH{2k+u-2}{\sC^{k-1}}
 \ar@{->}[r]^-{0} 
 &
 }
\]
We obtain the following commutative diagram with exact rows:
\[
 \xymatrixcolsep{4pc}
 \xymatrix{ 
 & 
 A_{u-3}^{\oplus{k+c+1\choose c-1}}
 \ar@{->}[d]_-{\HH{2k +u}{\psi^{k+2}}}
 \ar@{->}[r]^-{[\zeta_{u-2}^{k+1}]}
 & A_{u-2}^{\oplus{k+c \choose c-1}} 
 \ar@{=}[d] 
 \ar@{->}[r]^-{[\zeta_{u-1}^{k}]}
 &\im([\zeta_{u-1}^{k}]) \ar@{.>}[d]_{\cong}^{} 
 \ar@{->}[r] 
 & 0 
 \\
 0 
 \ar@{->}[r] 
 & 
 \HH{2k+u}{\sC^{k+1}} 
 \ar@{->}[r]_-{\HH{2k+u}{\sg^{k+1}}} 
 & A_{u-2}^{\oplus{k+c \choose c-1}} 
 \ar@{->}[r]_-{\HH{2k+u-1}{\psi^{k+1}}} 
 & 
 \HH{2k+u-1}{\sC^k} 
 \ar@{->}[r]
 & 0. 
 \\
 }
\]
By Snake Lemma $\HH{2k+u}{\psi^{k+2}}$ is surjective and one can show that $\Ker(\HH{2k+u}{\psi^{k+2}})=\Ker([\zeta_{u-2}^{k+1}])$. Therefore, $\HH{2k+u}{\sC^{k+1}} \cong \im([\zeta_{u-2}^{k+1}])$ for any $u\in\BZ$. 
By exactness of the first row in Figure~\ref{fig c}, with $j=k+1$, we get $\HH{2k+u+1}{\sff^{k+1}}=0$ for all $u\in\BZ$, hence the entire first long exact sequence splits. The statement of the lemma now holds for all integers $k \geq 0$ and $u \in \mathbb Z$.
\end{proof}

For our main result, we now describe the homology of the mapping cones when applied to the Koszul complex in Setup \ref{setup codepth c}. 

\begin{theorem}
\label{homology c}
Let $(R,\fm,\sk)$ be a complete intersection local ring of codepth $c \geq 1$. Let $K$ be the Koszul complex of $R$ on a minimal set of generators of $\fm$ and $A=A_0\oplus A_1\oplus \cdots \oplus A_c$ be its homology. Under the Setup \ref{setup codepth c}, for any $k \geq 1$, we have 
\[\HH{i}{\sC^k} \cong \begin{cases}
\HH{i}{\sC^{k-1}}, & \quad 0\leq i \leq 2k-2 \text{ or } i \geq 2k+c+1 \\
0, & \quad i=2k-1 \text{ or } i=2k \\
\im([\zeta_u^{k}]), & \quad \text{for all}\  i=2k+u,\ \text{where}\ 1\leq u \leq c,\\
\end{cases}\]
where $[\zeta^k]$ are maps in the following exact sequence in Proposition~\ref{prop:zeta c}:
\begin{equation*}
    0\to A_0^{\oplus{k+c\choose c-1}}\xrightarrow{[\zeta^{k}_1]}A_1^{\oplus{k+c-1\choose c-1}} \xrightarrow{[\zeta^{k-1}_2]}A_2^{\oplus{k+c-2\choose c-1}} \to \cdots \xrightarrow{[\zeta^{k-c+1}_{c}]}A_c^{\oplus{k\choose c-1}} \to  0.
\end{equation*}
\end{theorem}

\begin{proof}
From the iterated mapping cone construction in Definition~\ref{def cone}, see \eqref{Cj}, observe that we have the following equality of complexes: 
\[\sC^{k}_{\leq k-1}=\sC^{k-1}_{\leq k-1},\] 
By \eqref{cone diagram c}, we have the induced long exact sequence of homology:
\[ 
A_{i-2k+1}^{\oplus{k+c-1\choose c-1}}=\HH{i+1}{\Sigma^{2k}K^{\oplus{k+c-1\choose c-1}}}\xra{}\HH{i}{\sC^k}\to\HH{i}{\sC^{k-1}}\to\HH{i}{\Sigma^{2k}K^{\oplus{k+c-1\choose c-1}}}=A_{i-2k}^{\oplus{k+c-1\choose c-1}}. 
\]
Recall that $A_\ell = \HH{\ell}{K} = 0$ for all $\ell < 0$ or $\ell > c$. Here, $i-2k>c \iff i-2k \geq c+1 \iff i \geq 2k+c+1$. Similarly, $i-2k+1<0 \iff i-2k+1 \leq -1 \iff i \leq 2k-2$. Therefore, for $0\leq i \leq 2k-2$ or $i \geq 2k+c+1$, it follows from the above exact sequence that $\HH{i}{\sC^k} \cong \HH{i}{\sC^{k-1}}$. The rest of statement  follows directly from Lemma~\ref{imzeta}. We note that by \eqref{zeta codepth c}, $\zeta_0=0$ and $\zeta_{-1}=0$.
\end{proof}

\section{An Application of the Homology of the Iterated Mapping Cone}
\label{sec:application}

\subsection{Minimal Free Resolution of the Residue Field Over a Complete Intersection Ring}
\label{subsec:resolution}

For a complete intersection $(R,\fm,\sk)$ of embedding dimension $n$ and codepth $c \geq 1$, the minimal free resolution of the residue field $\sk$ over $R$ is well-known and studied by others by using Tate's construction \cite[Theorems 4 and 6]{T}. In this section, independent from the Tate's construction, we use the mapping cone construction and its homology in Theorem~\ref{homology c}, applied on a sequence of complexes built from the Koszul complex $K$, to explicitly describe this resolution and recover the formula of the Poincar\'e series of $\sk$ over $R$.

 \begin{theorem}
   \label{ci codepth c}
   Let $(R,\fm,\sk)$ be a complete intersection local ring of codepth $c \geq 1$. Using the notation from \eqref{zeta codepth c}, let $(F,\partial^F)$ be the complex over $R$ given by: 
 \[
     F_i=K_i^{\oplus{ c-1 \choose c-1}} \oplus K^{\oplus{c \choose c-1}}_{i-2}\oplus K^{\oplus{c+1 \choose c-1}}_{i-4}\oplus\cdots \oplus K^{\oplus{c+j-1\choose c-1}}_{i-2j}\oplus\cdots\quad 
 \]
    for all $i \geq 0$, and the differential $\partial_i^F\colon F_i\to F_{i-1}$ is given by:
     \[
 \xymatrixrowsep{3pc}
 \xymatrixcolsep{5pc}
 \xymatrix{
 &K^{\oplus{c-1 \choose c-1}}_{i\hspace{0.3cm}}
 \ar@{}[d]_{\bigoplus}
 \ar@{->}[r]^{\partial^{\oplus{c-1\choose c-1}}_{i}}
 &K^{\oplus{c-1 \choose c-1}}_{i-1}\ar@{}[d]_{\bigoplus} \\
 &K^{\oplus{c \choose c-1}}_{i-2}\ar@{}[d]_{\bigoplus}
 \ar@{->}[ur]^{\zeta^0_{i-1}}
 \ar@{->}[r]^{\partial^{\oplus{c \choose c-1}}_{i-2}}
 &K^{\oplus{c \choose c-1}}_{i-3}\ar@{}[d]_{\bigoplus}\\
 &K^{\oplus{c+1 \choose c-1}}_{i-4}\ar@{}[d]_{\vdots}
 \ar@{->}[ur]^{\zeta^1_{i-3}}
 \ar@{->}[r]^{\partial^{\oplus{c+1 \choose c-1}}_{i-4}}
 &K^{\oplus{c+1 \choose c-1}}_{i-5}\ar@{}[d]_{\vdots}\\
 &&\\
&K^{\oplus{c+j-1 \choose c-1}}_{i-2j}\ar@{}[d]_{\vdots}
\ar@{->}[ur]^{\zeta^{j-1}_{i-2j+1}}
\ar@{->}[r]^{\partial^{\oplus{c+j-1 \choose c-1}}_{i-2j}}
&K^{\oplus{c+j-1 \choose c-1}}_{i-2j-1}\ar@{}[d]_{\vdots} \\
&&
 }
\]
Then $F$ is a minimal free resolution of the residue field $\sk$ over $R$.
Moreover, the Poincar\'e series of the residue field $\sk$ is given by: $\displaystyle P^R_{\sk}(t)=\frac{(1+t)^n}{(1-t^2)^c}.$
\end{theorem}

\begin{proof}
Using Lemma \ref{zeta chain c}, we get $\partial_{i-1}^F\circ \partial_{i}^F=0$, and thus $\im(\partial_i^F) \subseteq \Ker(\partial_{i-1}^F)$ for all $i\geq 0$, so $F$ is a complex.
To show exactness of $F$ at each $i\geq 1$, we apply induction on $i$.
 The complex $F$ is exactly the iterated mapping cone complex $\sC$ defined in \eqref{eq cone}, see Setup \ref{setup codepth c}. In particular we have
\begin{equation*}
  F_i=\sC^j_i,\quad\text{for all}\quad j\geq \left\lceil\frac{i+1}{2}\right\rceil.
\end{equation*}
Observe that for $k=\left\lceil\frac{i}{2}\right\rceil+1$, we have $F_{i+1} = M^k_{i+1}$, $F_{i} = M^k_{i}$, and $F_{i-1} = M^k_{i-1}$. Therefore, it is enough to show  exactness of the following sequence for each $i \geq 1$:
\begin{equation*}
  \sC^{k}_{i+1}\to\sC^k_{i}\to\sC^k_{i-1}, \quad\text{for}\quad k=\left\lceil\frac{i}{2}\right\rceil+1,
\end{equation*}
that is, for each $i \geq 1$ we want to show 
$\HH{i}{\sC^k}=0$.
For $i$ is even or odd, this translates to showing that for all $j\geq 2$ we have the equalities:
$\HH{2j-3}{\sC^j}=0=\HH{2j-2}{\sC^j},$
which hold by Theorem~\ref{homology c}.

As remarked in Section~\ref{sec:Koszul}, all entries of the vectors $z_j \in K_1=R^n$ are in the maximal ideal $\fm$. By \eqref{zeta codepth c}, all the differential maps $\partial^F$ of the resolution $F$ have entries in $\fm$. Hence, $F$ is a minimal free resolution of $\sk$. 

For the Poincar\'e statement in the theorem, if we write $P^R_{\sk}(t)=\sum_{i=0}^{\infty}\beta_it^i,$ then we get
\[ 
\beta_i=\rank_R F_i=\sum_{j\geq 0}(\rank_R K_{i-2j}){j+c-1\choose c-1}=\sum_{j\geq 0}{n\choose i-2j}{j+c-1\choose c-1}.
\]
This is exactly the coefficient of $t^i$ in the series:
\begin{align*}
    \frac{(1+t)^n}{(1-t^2)^c}&=
    \Big(\sum_{\ell\geq 0}{n\choose \ell}t^\ell \Big)\cdot\Big(\sum_{j\geq 0}{j+c-1\choose c-1}t^{2j}\Big)
\end{align*}
which proves the desired equality.
\end{proof}

\subsection{DG Algebra Structure on the Minimal Free Resolution} 
\label{subsec:DGA}

Suppose $R$ is a complete intersection ring of embedding dimension $n$ and codepth $c$. By Tate's construction, the minimal free resolution of $\sk$ over $R$ has a differential graded (DG) algebra structure \cite[Theorem 1]{T}. For completeness, we describe here the DG algebra structure, which is naturally inherited from that of the Koszul complex $(K,\partial)$, of the minimal free resolution $F$ in Theorem~\ref{ci codepth c}

Let $\{e_1,\dots, e_n\}$ be a basis for the free $R$-module $K_1$. It suffices to describe the multiplication in $F$ on a basis. Consider a basis element $\bsa$ in $F_i$ that corresponds to a basis element in $K_{i-2j}^{\oplus{c+j-1\choose c-1}}$ that has the element $e_{s_1}\wedge \dots \wedge e_{s_{i-2j}}$ in the position $\ell_1\dots\ell_j$ in the tuple. Similarly, we consider a basis element $\bsb$ in $F_{i'}$ that corresponds to a basis element in $K_{i'-2j'}^{\oplus{c+j'-1\choose c-1}}$ that has the element $e_{s'_1}\wedge \dots \wedge e_{s'_{i'-2j'}}$ in the position $\ell'_1\dots\ell'_{j'}$ in the tuple. The product $\bsa\bsb$ is the basis element in $F_{i+i'}$ that corresponds to the basis element in $K_{(i+i')-2(j+j')}^{\oplus{c+(j+j')-1\choose c-1}}$ that has the element 
 $e_{s_1}\wedge \dots \wedge e_{s_{i-2j}} \wedge e_{s'_1}\wedge \dots \wedge e_{s'_{i'-2j'}}$
 in the position $[\ell_1\dots\ell_j\ell'_1\dots\ell'_{j'}]$. 

We check the differential maps $\partial^F$ respect the graded Leibniz rule:
\[\partial^F(\bsa\bsb) = \partial^F(\bsa) \, \bsb + (-1)^i \bsa \, \partial^F(\bsb),\]
by comparing the entries in corresponding positions. 
The differential map $\partial^F $ sends each element in two ways, via $\partial$ and via $\zeta$. 

\begin{align*}
&\text{Element}&&\text{Position}\\ \hline \\
&z_{\ell_r}\wedge e_{s_1}\wedge \dots\wedge e_{s_{i-2j}}\wedge e_{s'_1}\wedge \dots \wedge e_{s'_{i'-2j'}}
&& [\ell_1\dots \check{\ell_r}\dots \ell_j\ell'_1\dots\ell'_{j'}]\\
&\sum_{t=1}^{i-2j}(-1)^t x_{s_t} e_{s_1}\wedge \dots\wedge \check{e}_{s_t}\wedge\dots e_{s_{i-2j}} \wedge e_{s'_1}\wedge \dots \wedge e_{s'_{i'-2j'}}
&&[\ell_1\dots\ell_j\ell'_1\dots\ell'_{j'}]\\
&(-1)^i \, e_{s_1}\wedge \dots \wedge e_{s_{i-2j}} \wedge z_{\ell'_{r'}}\wedge e_{s'_1}\wedge \dots\wedge e_{s'_{i'-2j'}}
&& [\ell_1\dots \ell_j\ell'_1\dots \check{\ell}'_{r'}\dots\ell'_{j'}]\\
&(-1)^i \, e_{s_1}\wedge \dots \wedge e_{s_{i-2j}} \wedge \sum_{t'=1}^{i'-2j'}(-1)^{t'} x_{s'_t} e_{s'_1}\wedge \dots\wedge \check{e}_{s'_{t'}}\wedge\dots e_{s'_{i'-2j'}}  
&&  [\ell_1\dots\ell_j\ell'_1\dots\ell'_{j'}]\\
\hline \\
&z_{\ell_r}\wedge e_{s_1}\wedge \dots\wedge e_{s_{i-2j}}\wedge e_{s'_1}\wedge \dots \wedge e_{s'_{i'-2j'}}
&& [\ell_1\dots \check{\ell_r}\dots \ell_j\ell'_1\dots\ell'_{j'}]\\
& z_{\ell_{r'}}\wedge e_{s_1}\wedge \dots \wedge e_{s_{i-2j}} \wedge e_{s'_1}\wedge \dots\wedge e_{s'_{i'-2j'}}&& [\ell_1\dots \ell_j\ell'_1\dots \check{\ell}'_{r'}\dots\ell'_{j'}]\\
&\sum_{t=1}^{i-2j}(-1)^t x_{s_t} e_{s_1}\wedge \dots\wedge \check{e}_{s_t}\wedge\dots e_{s_{i-2j}} \wedge e_{s'_1}\wedge \dots \wedge e_{s'_{i'-2j'}}
&&[\ell_1\dots\ell_j\ell'_1\dots\ell'_{j'}]\\
&\sum_{t'=1}^{i'-2j'}(-1)^{i+t'} x_{s'_t}\, e_{s_1}\wedge \dots \wedge e_{s_{i-2j}} \wedge e_{s'_1}\wedge \dots\wedge \check{e}_{s'_{t'}}\wedge\dots e_{s'_{i'-2j'}}  &&[\ell_1\dots\ell_j\ell'_1\dots\ell'_{j'}]
\\
\hline 
\end{align*}

In the first block of expressions above, the first line with $1\leq r \leq j$  represents $\zeta(\bsa) \wedge \bsb$, the second line represents $\partial(\bsa) \wedge \bsb$, the third line with $1\leq r'\leq j'$ represents $(-1)^i \bsa \wedge \zeta(\bsb)$, and the fourth line represent $(-1)^i \bsa \wedge \partial(\bsb)$, where the notation $\check{(-)}$ means we remove the element from the expression. This first block together gives $\partial^F(\bsa)\, \bsb + (-1)^i \bsa \, \partial^F(\bsb)$. 

On the other hand, in the second block, for $1\leq r \leq j$ and $1\leq r'\leq j'$, the first and second lines represent $\zeta(\bsa\bsb)$, the third and fourth lines represent $\partial(\bsa\bsb)$. The second block together gives $\partial^F(\bsa\bsb)$. Now, comparing both blocks of expressions and noting that all the other components are zero, we see that the Leibniz rule holds with respect to $\partial^F$:
\[\partial^F(\bsa\bsb) = \partial^F(\bsa) \, \bsb + (-1)^i \bsa \, \partial^F(\bsb).\]
Hence, the defined multiplication gives a DG algebra structure on the resolution $F$ in Theorem~\ref{ci codepth c}.

\subsection{Examples}
\label{subsec:examples}

Since $K_i=0$ for $i<0$ and $i>c$, the module $F_i$ in Theorem~\ref{ci codepth c} is a finite direct sum of Koszul components. Using appropriate bases of $F_i$, the differential map $\partial_i^F$ can be expressed as block matrices. We illustrate these block maps through examples of complete intersection rings of codepth 2 and of codepth 3, and observe combinatorial and visual patterns for the differential maps $\partial_{\text{even}}^F$ and $\partial_{\text{odd}}^F$. Some computations here were verified by {\tt Macaulay2} \cite{M2}.

In the following examples, we consider a complete intersection ring $R=Q/I$, where $Q=\sk[x,y,z]_{(x,y,z)}$ with maximal ideal $\fn$ and  $I\subset\fn^2$, of embedding dimension three. The Koszul complex $(K,\partial)$ on the minimal set of generators $\{x,y,z\}$ of $\fn/I$ is given by
\[
0\to R\xra{\partial_3} R^{\oplus 3}\xra{\partial_2} R^{\oplus 3} \xra{\partial_1} R,
\]
where
\[
\partial_1=
\begin{pmatrix}
x&y&z    
\end{pmatrix}, \quad 
\partial_2=
\begin{pmatrix}
-y&-z& 0\\
 x& 0& -z\\
 0& x&  y
\end{pmatrix}, \quad
\partial_3=
\begin{pmatrix}
z\\
-y\\
x
\end{pmatrix}
\]
with respect to the following bases for $K$:
\begin{align*}
\{1\}, &\qquad\text{for}\ K_0=R\\
\{e_1,\ e_2,\ e_3\}, &\qquad\text{for}\ K_1=R^{\oplus 3}\\    
\{e_{12},\ e_{13},\ e_{23}\}, &\qquad\text{for}\ K_2=R^{\oplus 3}\\ 
\{e_{123}\}, &\qquad\text{for}\ K_3=R,   
\end{align*}
where $e_{ij} = e_i \wedge e_j$ for $1 \leq i,j,k \leq 3$, and similarly $e_{123}=e_1\wedge e_2\wedge e_3$.    
We visually express the differential  maps as blocks and will use them in the desired minimal resolution of $\sk$:
\[\partial_1 = \begin{tikzpicture}[scale=0.25]
\filldraw[draw=black,color=black!20!white] (0,0) rectangle (3,1);
\end{tikzpicture}, \qquad \qquad
\partial_2 = \begin{tikzpicture}[scale=0.25]
\filldraw[draw=black,color=black!50!white] (0,0) rectangle (3,3);
\end{tikzpicture}, \qquad \qquad 
\partial_3 = \begin{tikzpicture}[scale=0.25]
\filldraw[draw=black,color=black!80!white] (0,0) rectangle (1,3);
\end{tikzpicture}.
\]

\bigskip

\noindent
{\bf Codepth 2 Example.}  If $(R,\fm,\sk)$ is  a complete intersection of codepth two, then by Theorem~\ref{ci codepth c} a minimal free resolution $F$ of the residue field $\sk$ over $R$ has: 
\[
 F_i=K_i
     \oplus K^{\oplus{2}}_{i-2}
      \oplus K^{\oplus{3}}_{i-4}\oplus\cdots 
      \oplus K^{\oplus{j+1}}_{i-2j}\oplus\cdots,  \qquad \text{ for all } i\geq 0.
\]

In the following example, we illustrate $\zeta_u^k$, $\partial_i$ and $\partial_i^F$ in Theorem~\ref{ci codepth c} as block matrices.

\begin{example} 
\label{exp:codepth2 edim3}
Consider the complete intersection  ring $R=Q/(x^2, y^2+z^2)$ of embedding dimension three and codepth two. The homology $A_1$ is generated by the classes of the following elements in $K_1$:
\[
z_1=
\begin{pmatrix}
  x\\
  0\\
  0
\end{pmatrix}=xe_1
\qquad
\text{and}
\qquad
z_2=\begin{pmatrix}
    0\\
    y\\
    z
\end{pmatrix}=ye_2+ze_3.
\]
For each integer $k \geq 0$ and $1\leq u\leq 3$, we build $\zeta_u^{k}\colon K_{u-1}^{\oplus{k+2}}\to K_u^{\oplus{k+1}}$ inductively from $\zeta_u^0\colon K_{u-1}^{\oplus 2} \to K_u$ as follows. By \eqref{zeta codepth c}, $\zeta_u^0$ is given by $\begin{pmatrix} z_1 \wedge - & z_2 \wedge - \end{pmatrix}$ over the standard basis of $K_u=R^{\oplus {3 \choose u}}$. 

For $u=1$, the columns of $\zeta_1^0\colon R^{\oplus 2} \to R^{\oplus 3}$ are exactly the vectors $z_1$ and $z_2$, so it is:
\[
\zeta_1^0=
\begin{pmatrix}
     x& 0 \\
     0& y\\
     0& z
\end{pmatrix}.\]
For simplicity, we denote $\zeta_1^0$ as a $3 \times 2$ block. Then $\zeta_1^1$ as a $6 \times 3$ block and $\zeta_1^2$ as a $9 \times 4$ block are:
\[
\zeta_1^0 = \begin{tikzpicture}[scale=0.25]
\draw[pattern=dots, pattern color=black] (0,0) rectangle (2,3);
\end{tikzpicture}, 
\qquad \qquad 
\zeta_1^1 = 
\begin{tikzpicture}[scale=0.25]
\draw (0,0) rectangle (3,-6);
\draw[pattern=dots, pattern color=black] (0,0) rectangle (2,-3);
\draw[pattern=dots, pattern color=black] (1,-3) rectangle (3,-6);
\end{tikzpicture}, 
\qquad \qquad 
\zeta_1^2 = 
\begin{tikzpicture}[scale=0.25]
\draw (0,0) rectangle (4,-9);
\draw[pattern=dots, pattern color=black] (0,0) rectangle (2,-3);
\draw[pattern=dots, pattern color=black] (1,-3) rectangle (3,-6);
\draw[pattern=dots, pattern color=black] (2,-6) rectangle (4,-9);
\end{tikzpicture}.
\]
The pattern for $\zeta_1^k$ continues this way, where the blank entries are filled with $0$ and each $3 \times 2$ block is shifted to the right by one entry. 

For $u=2$, the columns of  $\zeta_2^0\colon (R^{\oplus 3})^{\oplus 2} \to R^{\oplus 3}$ are given by 
$
z_1\wedge e_1,z_1\wedge e_2,z_1\wedge e_3; \,
z_2\wedge e_1,z_2\wedge e_2,z_2\wedge e_3.
$
Using the computations
\begin{align*}
  z_1\wedge e_1 &= 0&   
  z_1\wedge e_2 &= xe_{12}&   
  z_1\wedge e_3 &= xe_{13}&   
  \\
  z_2\wedge e_1 &= -ye_{12}-ze_{13}& 
  z_2\wedge e_2 &= -ze_{23}&   
  z_2\wedge e_3 &= ye_{23},&  
\end{align*}
we get:
\[
\zeta_2^0=
\left(\begin{array}{ccc|ccc}
    0&x &0& -y& 0& 0\\
    0&0 &x& -z& 0& 0\\
    0&0 &0& 0& -z& y
\end{array}\right).
\]
For simplicity, we denote $\zeta_2^0$ as a $3\times 6$ block. Then $\zeta_2^1$ as a $6\times 9$ block and $\zeta_2^2$ as a $9 \times 12$ block are:
\[
\zeta_2^0 = \begin{tikzpicture}[scale=0.25]
\draw[pattern=north east lines, pattern color=black] (0,0) rectangle (6,-3);
\end{tikzpicture}, 
\qquad  
\zeta_2^1 = 
\begin{tikzpicture}[scale=0.25]
\draw (0,0) rectangle (9,-6);
\draw[pattern=north east lines, pattern color=black] (0,0) rectangle (6,-3);
\draw[pattern=north east lines, pattern color=black] (3,-3) rectangle (9,-6);
\end{tikzpicture}, 
\qquad  
\zeta_2^2 = 
\begin{tikzpicture}[scale=0.25]
\draw (0,0) rectangle (12,-9);
\draw[pattern=north east lines, pattern color=black] (0,0) rectangle (6,-3);
\draw[pattern=north east lines, pattern color=black] (3,-3) rectangle (9,-6);
\draw[pattern=north east lines, pattern color=black] (6,-6) rectangle (12,-9);
\end{tikzpicture}. 
\]
The pattern for $\zeta_2^k$ continues this way, where the blank entries are filled with $0$ and each $3 \times 6$ block is shifted to the right by three entries. 

For $u=3$, the columns of  
$\zeta_3^0\colon (R^{\oplus 3})^{\oplus 2} \to R$ 
are given by 
$
z_1\wedge e_{12},z_1\wedge e_{13},z_1\wedge e_{23}; \,
z_2\wedge e_{12},z_2\wedge e_{13},z_2\wedge e_{23}. 
$
Using the computations
\begin{align*}
  z_1\wedge e_{12} &= 0&   
  z_1\wedge e_{13} &= 0&   
  z_1\wedge e_{23} &= xe_{123}  
  \\
  z_2\wedge e_{12} &= ze_{123}&   
  z_2\wedge e_{13} &= -ye_{123}&   
  z_2\wedge e_{23} &= 0,
\end{align*}
we get:
\[
\zeta_3^0=
\left(
\begin{array}{ccc|ccc}
   0&0&x &z&-y&0 
\end{array}\right).
\]
For simplicity, we denote $\zeta_3^0$ as a $1 \times 6$ block. Then $\zeta_3^1$ as a $2\times 9$ block and $\zeta_3^2$ as a $3 \times 12$ block are:
\[
\zeta_3^0 = \begin{tikzpicture}[scale=0.25]
\draw[pattern=crosshatch, pattern color= gray] (0,0) rectangle (6,-1);
\end{tikzpicture}, 
\quad  
\zeta_3^1 = 
\begin{tikzpicture}[scale=0.25]
\draw (0,0) rectangle (9,-2);
\draw[pattern=crosshatch, pattern color= gray] (0,0) rectangle (6,-1);
\draw[pattern=crosshatch, pattern color= gray] (3,-1) rectangle (9,-2);
\end{tikzpicture}, 
\quad
\zeta_3^2 = 
\begin{tikzpicture}[scale=0.25]
\draw (0,0) rectangle (12,-3);
\draw[pattern=crosshatch, pattern color= gray] (0,0) rectangle (6,-1);
\draw[pattern=crosshatch, pattern color= gray] (3,-1) rectangle (9,-2);
\draw[pattern=crosshatch, pattern color= gray] (6,-2) rectangle (12,-3);
\end{tikzpicture}. 
\]
The pattern for $\zeta_3^k$ continues this way, where the blank entries are filled with $0$ and each $1 \times 6$ block is shifted to the right by three entries.

Now, we put together the above block maps and express the differential maps $\partial_i^F$ in the minimal free resolution $F$ of $\sk$ over $R$ as blocks, where the blank entries are filled with 0:
\begin{align*}
\partial^F_1 &=  \begin{tikzpicture}[scale=0.25]
\filldraw[draw=black,color=black!20!white] (0,0) rectangle (3,1);
\end{tikzpicture},
&
\partial^F_2&= 
\begin{tikzpicture}[scale=0.25]
\filldraw[draw=black,color=black!50!white] (0,0) rectangle (3,3);
\draw[pattern=dots, pattern color=black] (3,0) rectangle (5,3);
\end{tikzpicture},
\\
\partial^F_3&= \begin{tikzpicture}[scale=0.25]
\filldraw[draw=black,color=black!80!white] (0,0) rectangle (1,-3);
\draw[pattern=north east lines, pattern color=black] (1,0) rectangle (7,-3);
\draw(0,-3) rectangle (1,-5);
\draw(1,-3) rectangle (7,-5);
 \filldraw[draw=black,color=black!20!white] (1,-3) rectangle (4,-4);
 \filldraw[draw=black,color=black!20!white] (4,-4) rectangle (7,-5);
\end{tikzpicture},
&
\partial^F_4&= \begin{tikzpicture}[scale=0.25]
\draw[pattern=crosshatch, pattern color= gray] (0,0) rectangle (6,-1);
\draw(6,0) rectangle (9,-1);
\filldraw[draw=black,color=black!50!white] (0,-1) rectangle (3,-4);
\filldraw[draw=black,color=black!50!white] (3,-4) rectangle (6,-7);
\draw (0,-1) rectangle (6,-7);
\draw[pattern=dots, pattern color=black] (6,-1) rectangle (8,-4);
\draw[pattern=dots, pattern color=black] (7,-4) rectangle (9,-7);
\draw (6,-1) rectangle (9,-7);
\end{tikzpicture},
\\
\partial^F_5 &=
\begin{tikzpicture}[scale=0.25]
\draw(0,0) rectangle (2,-6);
 \filldraw[draw=black,color=black!80!white] (0,0) rectangle (1,-3);
 \filldraw[draw=black,color=black!80!white] (1,-3) rectangle (2,-6);
\draw(0,-6) rectangle (2,-9); 
\draw(2,0) rectangle (11,-6);
 \draw[pattern=north east lines, pattern color=black] (2,0) rectangle (8,-3);
 \draw[pattern=north east lines, pattern color=black] (5,-3) rectangle (11,-6);
\draw(2,-6) rectangle (11,-9);
 \filldraw[draw=black,color=black!20!white] (2,-6) rectangle (5,-7);
 \filldraw[draw=black,color=black!20!white] (5,-7) rectangle (8,-8);
 \filldraw[draw=black,color=black!20!white] (8,-8) rectangle (11,-9);
\end{tikzpicture},
&
\partial^F_6 &=
\begin{tikzpicture}[scale=0.25]
\draw(0,0) rectangle (9,-2);
 \draw[pattern=crosshatch, pattern color= gray] (0,0) rectangle (6,-1);
 \draw[pattern=crosshatch, pattern color= gray] (3,-1) rectangle (9,-2);
\draw(9,0) rectangle (13,-2);
\draw(0,-2) rectangle (9,-11);
 \filldraw[draw=black,color=black!50!white] (0,-2) rectangle (3,-5);
 \filldraw[draw=black,color=black!50!white] (3,-5) rectangle (6,-8);
 \filldraw[draw=black,color=black!50!white] (6,-8) rectangle (9,-11);
 \draw (9,-2) rectangle (13,-11);
 \draw[pattern=dots, pattern color=black] (9,-2) rectangle (11,-5);
 \draw[pattern=dots, pattern color=black] (10,-5) rectangle (12,-8);
 \draw[pattern=dots, pattern color=black] (11,-8) rectangle (13,-11);
\end{tikzpicture}.
\end{align*}
\end{example}

Observe that for codepth 2, the differential maps $\partial^F_i$ for $i\geq 3$ can be built from two sets of matrices: $\{\partial_1,\partial_2,\partial_3\}$ which is independent of the ideal $I$,  and $\{\zeta_1^0, \zeta_2^0,\zeta_3^0\}$, as follows:
\begin{itemize}
    \item[--] When $i=2k+3$, we need $\partial_1$, $\partial_3$, and $\zeta_2^0$ as building blocks: $\partial_1$ appears diagonally $(k+2)$ times, $\partial_3$ appears diagonally $(k+1)$ times, and $\zeta_2^k$ is built from $\zeta_2^0$ $(k+1)$ times, where each $\zeta_2^0$ is shifted to the right three units.
    \item[--] When $i=2k+4$, we need $\partial_2$, $\zeta_1^0$, and $\zeta_3^0$ as building blocks, with $\partial_2$ appears diagonally $(k+2)$ times, $\zeta_1^k$ is built from $\zeta_1^0$ $(k+2)$ times, where each $\zeta_1^0$ is shifted to the right one unit, and $\zeta_3^k$ is built from $\zeta_3^0$ $(k+1)$ times, where each $\zeta_3^0$ is shifted to the right three units.
\end{itemize}

\bigskip

\noindent
{\bf Codepth 3 Example.}  If $(R,\fm,\sk)$ is a complete intersection of codepth three, then by Theorem~\ref{ci codepth c} a minimal free resolution $F$ of the residue field $\sk$ over $R$ has: 
\[F_i=K_i
     \oplus K^{\oplus{3}}_{i-2}
     \oplus K^{\oplus{6}}_{i-4}\oplus\cdots 
     \oplus K^{\oplus{j+2\choose 2}}_{i-2j}\oplus\cdots, \qquad \text{ for all } i\geq 0.
\]

\begin{remark}
\label{rem:spread} 
In the case $\codepth R=3=\edim R$, for each integer $k \geq 0$ and $1\leq u\leq 3$, by using the standard basis of $K_u=R^{\oplus{3\choose u}}$, we identify the map $\zeta_u^k \colon K_{u-1}^{\oplus{k+3 \choose 2}} \to K_u^{\oplus{k+2 \choose 2}}$ with its corresponding matrix of size ${3\choose u}{k+2\choose 2} \times {3\choose u-1}{k+3\choose 2}$. It can be built from the matrix corresponding to $\zeta_u^0$ as follows. 
\begin{itemize}
 \item[--] By \eqref{zeta codepth c}, $\zeta_u^0$ is given by $\begin{pmatrix} z_1 \wedge - & z_2 \wedge - & z_3 \wedge - \end{pmatrix}$ over the standard basis of $K_u=R^{\oplus {3 \choose u}}$. 
 \item[--] For each $1 \leq u \leq 3$, the ${3\choose u} \times 3{3\choose u-1}$ matrix  $\zeta_u^0$ is broken into two matrices:  
 
 \hspace{0.2in} $(\zeta_u^0)_1$ of size ${3\choose u} \times {3\choose u-1}$ that corresponds to $(z_1 \wedge -)$ , and \vspace{0.05in}
 
 \hspace{0.2in} $(\zeta_u^0)_2$ of size ${3\choose u} \times 2{3\choose u-1}$ that corresponds to $(z_2 \wedge - \quad z_3 \wedge -)$.
 \item[--] Let $(\zeta_u^0)'$ be the matrix of all zeros of the same size as $(\zeta_u^0)_1$.
 \item[--] For each $k\geq 1$, we construct a new matrix $\widetilde{\zeta}_u^k$ from $\zeta_u^0$ by inserting $k$ blocks of $(\zeta_u^0)'$ horizontally in between the two blocks $(\zeta_u^0)_1$ and $(\zeta_u^0)_2$.
 \item[--] For $k \geq 1$, each $\zeta_u^k$ is constructed inductively from $\zeta_u^{k-1}$ as follows: start with $\zeta_u^{k-1}$, then stack underneath it  $k+1$ blocks of $\widetilde{\zeta}_u^k$, where each block $\widetilde{\zeta}_u^k$ is shifted to the right by ${3\choose u-1}$ entries.
 \item[--] All the blank entries in the matrix $\zeta_u^k$ are filled with zeros. 
\end{itemize}
\end{remark}

In the following example, we illustrate $\zeta_u^k$, $\partial_i$ and $\partial_i^F$ in Theorem~\ref{ci codepth c} as block matrices.

\begin{example} 
\label{exp:codepth3 edim3}
Consider the complete intersection ring $R=Q/(x^2+y^2, xz, z^2+xy)$ of embedding dimension three and codepth three. The homology $A_1$ is generated by the classes of the following elements in $K_1$:
\[
z_1=
\begin{pmatrix}
  x\\
  y\\
  0
\end{pmatrix}=xe_1+ye_2,
\quad
z_2=\begin{pmatrix}
    0\\
    0\\
    x
\end{pmatrix}=xe_3,
\quad
\text{and}
\quad
z_3=\begin{pmatrix}
    0\\
    x\\
    z
\end{pmatrix}=xe_2+ze_3.
\]

For $u=1$, the columns of $\zeta_1^0\colon R^{\oplus 2} \to R^{\oplus 3}$ are exactly the vectors $z_1, z_2$, and $z_3$ so it is:
\[
\zeta_1^0=
\begin{pmatrix}
     x& 0 & 0\\
     y& 0 & x\\
     0& x & z
\end{pmatrix}. \]
For simplicity, we denote $\zeta_1^0$ as a $3 \times 3$ block. Then $\zeta_1^1$ as a $9 \times 6$ block and $\zeta_1^2$ as a $18 \times 10$ block are:
\[
\zeta_1^0 = \begin{tikzpicture}[scale=0.25]
\draw[pattern=dots, pattern color=black] (0,0) rectangle (3,3);
\end{tikzpicture} 
= \begin{tikzpicture}[scale=0.25]
\draw[pattern=dots, pattern color=black] (0,0) rectangle (3,3);
\draw (1,0) -- (1,3);
\end{tikzpicture}, 
\qquad \qquad 
\zeta_1^1 = 
\begin{tikzpicture}[scale=0.25]
\draw (0,0) rectangle (6,-9);
 \draw[pattern=dots, pattern color=black] (0,0) rectangle (3,-3);
 \draw (1,0) -- (1,-3);
 \draw[pattern=dots, pattern color=black] (1,-3) rectangle (2,-6);
 \draw (2,-3) rectangle (3,-6);
 \draw[pattern=dots, pattern color=black] (3,-3) rectangle (5,-6);
 \draw[pattern=dots, pattern color=black] (2,-6) rectangle (3,-9);
 \draw (3,-6) rectangle (4,-9);
 \draw[pattern=dots, pattern color=black] (4,-6) rectangle (6,-9);
\end{tikzpicture}, 
\qquad \qquad 
\zeta_1^2 = 
\begin{tikzpicture}[scale=0.25]
\draw (0,0) rectangle (10,-18);
 \draw[pattern=dots, pattern color=black] (0,0) rectangle (3,-3);
 \draw (1,0) -- (1,-3);
 \draw[pattern=dots, pattern color=black] (1,-3) rectangle (2,-6);
 \draw (2,-3) rectangle (3,-6);
 \draw[pattern=dots, pattern color=black] (3,-3) rectangle (5,-6);
 \draw[pattern=dots, pattern color=black] (2,-6) rectangle (3,-9);
 \draw (3,-6) rectangle (4,-9);
 \draw[pattern=dots, pattern color=black] (4,-6) rectangle (6,-9);
 \draw[pattern=dots, pattern color=black] (3,-9) rectangle (4,-12);
 \draw (4,-9) rectangle (6,-12);
 \draw[pattern=dots, pattern color=black] (6,-9) rectangle (8,-12);
 \draw[pattern=dots, pattern color=black] (4,-12) rectangle (5,-15);
 \draw (5,-12) rectangle (7,-15);
 \draw[pattern=dots, pattern color=black] (7,-12) rectangle (9,-15);
 \draw[pattern=dots, pattern color=black] (5,-15) rectangle (6,-18);
 \draw (6,-15) rectangle (8,-18);
 \draw[pattern=dots, pattern color=black] (8,-15) rectangle (10,-18);
\end{tikzpicture}. 
\]
The pattern for $\zeta_1^k$ continues this way, where the blank entries are filled with $0$; see Remark \ref{rem:spread}.

For $u=2$, the columns of  $\zeta_2^0\colon (R^{\oplus 3})^{\oplus 3} \to R^{\oplus 3}$ are given by 
$
z_1\wedge e_1,z_1\wedge e_2,z_1\wedge e_3; \,
z_2\wedge e_1,z_2\wedge e_2,z_2\wedge e_3; \,
z_3\wedge e_1,z_3\wedge e_2,z_3\wedge e_3.
$
Using the computations
\begin{align*}
  z_1\wedge e_1 &= -ye_{12}&   
  z_1\wedge e_2 &= xe_{12}&   
  z_1\wedge e_3 &= xe_{13}+ye_{23}&   
  \\
  z_2\wedge e_1 &= -xe_{13}& 
  z_2\wedge e_2 &= -xe_{23}&   
  z_2\wedge e_3 &= 0&  
  \\
  z_3\wedge e_1 &= -xe_{12}-ze_{13}& 
  z_3\wedge e_2 &= -ze_{23}&   
  z_3\wedge e_3 &= xe_{23},&  
\end{align*}
we get:
\[
\zeta_2^0=
\left(\begin{array}{ccc|ccc|ccc}
   -y&x &0& 0& 0& 0&-x& 0&0\\
    0&0 &x&-x& 0& 0&-z& 0&0\\
    0&0 &y& 0&-x& 0& 0&-z&x
\end{array}\right).
\]

For simplicity, we denote $\zeta_2^0$ as a $3\times 9$ block. Then $\zeta_2^1$ as a $9\times 18$ block and $\zeta_2^2$ as a $18 \times 30$ block are:
\[
\zeta_2^0
= \begin{tikzpicture}[scale=0.25]
\draw[pattern=north east lines, pattern color=black] (0,0) rectangle (9,-3);
\end{tikzpicture}
=
\begin{tikzpicture}[scale=0.25]
\draw[pattern=north east lines, pattern color=black] (0,0) rectangle (9,-3);
\draw (3,0) -- (3,-3);
\end{tikzpicture}, 
\qquad  
\zeta_2^1 = 
\begin{tikzpicture}[scale=0.25]
\draw (3,0) -- (3,-3);
\draw (0,0) rectangle (18,-9);
 \draw[pattern=north east lines, pattern color=black] (0,0) rectangle (9,-3);
 \draw[pattern=north east lines, pattern color=black] (3,-3) rectangle (6,-6);
 \draw (6,-3) rectangle (9,-6);
 \draw[pattern=north east lines, pattern color=black] (9,-3) rectangle (15,-6);
 \draw[pattern=north east lines, pattern color=black] (6,-6) rectangle (9,-9);
 \draw (9,-6) rectangle (12,-9);
 \draw[pattern=north east lines, pattern color=black] (12,-6) rectangle (18,-9);
\end{tikzpicture}, \]
\[  
\zeta_2^2 = 
\begin{tikzpicture}[scale=0.25]
\draw (3,0) -- (3,-3);
\draw (0,0) rectangle (30,-18);
\draw[pattern=north east lines, pattern color=black] (0,0) rectangle (9,-3);
 \draw[pattern=north east lines, pattern color=black] (3,-3) rectangle (6,-6);
 \draw (6,-3) rectangle (9,-6);
 \draw[pattern=north east lines, pattern color=black] (9,-3) rectangle (15,-6);
 \draw[pattern=north east lines, pattern color=black] (6,-6) rectangle (9,-9);
 \draw (9,-6) rectangle (12,-9);
 \draw[pattern=north east lines, pattern color=black] (12,-6) rectangle (18,-9);
 \draw[pattern=north east lines, pattern color=black] (9,-9) rectangle (12,-12);
 \draw (12,-9) rectangle (18,-12);
 \draw[pattern=north east lines, pattern color=black] (18,-9) rectangle (24,-12);
 \draw[pattern=north east lines, pattern color=black] (12,-12) rectangle (15,-15);
 \draw (15,-12) rectangle (21,-15);
 \draw[pattern=north east lines, pattern color=black] (21,-12) rectangle (27,-15);
 \draw[pattern=north east lines, pattern color=black] (15,-15) rectangle (18,-18);
 \draw (18,-15) rectangle (24,-18);
 \draw[pattern=north east lines, pattern color=black] (24,-15) rectangle (30,-18);
\end{tikzpicture}.
\]
The pattern for $\zeta_2^k$ continues this way, where the blank entries are filled with $0$; see Remark \ref{rem:spread}. 

For $u=3$, the columns of  
$\zeta_3^0\colon (R^{\oplus 3})^{\oplus 3} \to R$ 
are given by 
$
z_1\wedge e_{12},z_1\wedge e_{13},z_1\wedge e_{23}; \,
z_2\wedge e_{12},z_2\wedge e_{13},z_2\wedge e_{23};\,
z_3\wedge e_{12},z_3\wedge e_{13},z_3\wedge e_{23}.
$
Using the computations
\begin{align*}
  z_1\wedge e_{12} &= 0&   
  z_1\wedge e_{13} &= -ye_{123}&   
  z_1\wedge e_{23} &= xe_{123}  
  \\
  z_2\wedge e_{12} &= xe_{123}&   
  z_2\wedge e_{13} &= 0&   
  z_2\wedge e_{23} &= 0
  \\
  z_3\wedge e_{12} &= ze_{123}&   
  z_3\wedge e_{13} &= -xe_{123}&   
  z_3\wedge e_{23} &= 0,
\end{align*}
we get:
\[
\zeta_3^0=
\left(
\begin{array}{ccc|ccc|ccc}
   0&-y&x &x&0&0 &z&-x&0 
\end{array}\right).
\]
For simplicity, we denote $\zeta_3^0$ as a $1 \times 9$ block. Then $\zeta_3^1$ as a $3\times 18$ block and $\zeta_3^2$ as a $6 \times 30$ block are:
\[
\zeta_3^0 = \begin{tikzpicture}[scale=0.25]
\draw[pattern=crosshatch, pattern color= gray] (0,0) rectangle (9,-1);
\end{tikzpicture} 
= \begin{tikzpicture}[scale=0.25]
\draw (3,0) -- (3,-1);
\draw[pattern=crosshatch, pattern color= gray] (0,0) rectangle (9,-1);
\end{tikzpicture}, 
\qquad  
\zeta_3^1 = 
\begin{tikzpicture}[scale=0.25]
\draw (3,0) -- (3,-1);
\draw (0,0) rectangle (18,-3);
 \draw[pattern=crosshatch, pattern color= gray] (0,0) rectangle (9,-1);
 \draw[pattern=crosshatch, pattern color= gray] (3,-1) rectangle (6,-2);
 \draw (6,-1) rectangle (9,-2);
 \draw[pattern=crosshatch, pattern color= gray] (9,-1) rectangle (15,-2);
 \draw[pattern=crosshatch, pattern color= gray] (6,-2) rectangle (9,-3);
 \draw (9,-2) rectangle (12,-3);
 \draw[pattern=crosshatch, pattern color= gray] (12,-2) rectangle (18,-3);
\end{tikzpicture}, \]
\[
\zeta_3^2 = 
\begin{tikzpicture}[scale=0.25]
\draw (3,0) -- (3,-1);
\draw (0,0) rectangle (30,-6);
 \draw[pattern=crosshatch, pattern color= gray] (0,0) rectangle (9,-1);
 \draw[pattern=crosshatch, pattern color= gray] (3,-1) rectangle (6,-2);
 \draw (6,-1) rectangle (9,-2);
 \draw[pattern=crosshatch, pattern color= gray] (9,-1) rectangle (15,-2);
 \draw[pattern=crosshatch, pattern color= gray] (6,-2) rectangle (9,-3);
 \draw (9,-2) rectangle (12,-3);
 \draw[pattern=crosshatch, pattern color= gray] (12,-2) rectangle (18,-3);
 \draw[pattern=crosshatch, pattern color= gray] (9,-3) rectangle (12,-4);
 \draw (12,-3) rectangle (18,-4);
 \draw[pattern=crosshatch, pattern color= gray] (18,-3) rectangle (24,-4);
 \draw[pattern=crosshatch, pattern color= gray] (12,-4) rectangle (15,-5);
 \draw (15,-4) rectangle (21,-5);
 \draw[pattern=crosshatch, pattern color= gray] (21,-4) rectangle (27,-5);
 \draw[pattern=crosshatch, pattern color= gray] (15,-5) rectangle (18,-6);
 \draw (18,-5) rectangle (24,-6);
 \draw[pattern=crosshatch, pattern color= gray] (24,-5) rectangle (30,-6);
\end{tikzpicture}.
\]
The pattern for $\zeta_3^k$ continues this way, where the blank entries are filled with $0$; see Remark \ref{rem:spread}.

Now, we put together the above block maps and express the differential maps $\partial_i^F$ in the minimal free resolution $F$ of $\sk$ over $R$ as blocks, where the blank entries are filled with 0:
\begin{align*}
\partial^F_1 &=  \begin{tikzpicture}[scale=0.25]
\filldraw[draw=black,color=black!20!white] (0,0) rectangle (3,1);
\end{tikzpicture},
&
\partial^F_2&= 
\begin{tikzpicture}[scale=0.25]
\filldraw[draw=black,color=black!50!white] (0,0) rectangle (3,3);
\draw[pattern=dots, pattern color=black] (3,0) rectangle (6,3);
\end{tikzpicture},
\\
\partial^F_3&= \begin{tikzpicture}[scale=0.25]
\filldraw[draw=black,color=black!80!white] (0,0) rectangle (1,-3);
\draw[pattern=north east lines, pattern color=black] (1,0) rectangle (10,-3);
\draw(0,-3) rectangle (1,-6);
\draw(1,-3) rectangle (10,-6);
 \filldraw[draw=black,color=black!20!white] (1,-3) rectangle (4,-4);
 \filldraw[draw=black,color=black!20!white] (4,-4) rectangle (7,-5);
 \filldraw[draw=black,color=black!20!white] (7,-5) rectangle (10,-6);
\end{tikzpicture},
&
\partial^F_4&= \begin{tikzpicture}[scale=0.25]
\draw[pattern=crosshatch, pattern color= gray] (0,0) rectangle (9,-1);
\draw (10,-1)-- (10,-4);
\draw (9,0) rectangle (15,-1);
\draw (0,-1) rectangle (9,-10);
 \filldraw[draw=black,color=black!50!white] (0,-1) rectangle (3,-4);
 \filldraw[draw=black,color=black!50!white] (3,-4) rectangle (6,-7);
 \filldraw[draw=black,color=black!50!white] (6,-7) rectangle (9,-10);
\draw (9,-1) rectangle (15,-10);
 \draw[pattern=dots, pattern color=black] (9,-1) rectangle (12,-4);
 \draw[pattern=dots, pattern color=black] (10,-4) rectangle (11,-7); 
 \draw (11,-4) rectangle (12,-7);  \draw[pattern=dots, pattern color=black] (12,-4) rectangle (14,-7);
 \draw[pattern=dots, pattern color=black] (11,-7) rectangle (12,-10);
 \draw (12,-7) rectangle (15,-10);
 \draw[pattern=dots, pattern color=black] (13,-7) rectangle (15,-10);
\end{tikzpicture},
\\
\partial^F_5 &=
\begin{tikzpicture}[scale=0.25]
\draw(0,0) rectangle (3,-9);
 \filldraw[draw=black,color=black!80!white] (0,0) rectangle (1,-3);
 \filldraw[draw=black,color=black!80!white] (1,-3) rectangle (2,-6);
 \filldraw[draw=black,color=black!80!white] (2,-6) rectangle (3,-9);
\draw (3,0) rectangle (21,-9);
 \draw[pattern=north east lines, pattern color=black] (3,0) rectangle (12,-3);
 \draw (6,0) -- (6,-3);
 \draw[pattern=north east lines, pattern color=black] (6,-3) rectangle (9,-6);
\draw (9,-3) rectangle (12,-6);
 \draw[pattern=north east lines, pattern color=black] (12,-3) rectangle (18,-6);
 \draw[pattern=north east lines, pattern color=black] (9,-6) rectangle (12,-9);
\draw (12,-6) rectangle (15,-9);
 \draw[pattern=north east lines, pattern color=black] (15,-6) rectangle (21,-9);
\draw(0,-9) rectangle (3,-15); 
\draw(3,-9) rectangle (21,-15);
 \filldraw[draw=black,color=black!20!white] (3,-9) rectangle (6,-10);
 \filldraw[draw=black,color=black!20!white] (6,-10) rectangle (9,-11);
 \filldraw[draw=black,color=black!20!white] (9,-11) rectangle (12,-12);
 \filldraw[draw=black,color=black!20!white] (12,-12) rectangle (15,-13);
 \filldraw[draw=black,color=black!20!white] (15,-13) rectangle (18,-14);
 \filldraw[draw=black,color=black!20!white] (18,-14) rectangle (21,-15);
\end{tikzpicture},
&
\partial^F_6 &=
\begin{tikzpicture}[scale=0.25]
\draw (3,0)-- (3,-1);
\draw (19,-3)-- (19,-6);
\draw (0,0) rectangle (18,-3);
 \draw[pattern=crosshatch, pattern color= gray] (0,0) rectangle (9,-1);
 \draw[pattern=crosshatch, pattern color= gray] (3,-1) rectangle (6,-2);
 \draw (6,-1) rectangle (9,-2);
 \draw[pattern=crosshatch, pattern color= gray] (9,-1) rectangle (15,-2);
 \draw[pattern=crosshatch, pattern color= gray] (6,-2) rectangle (9,-3);
 \draw (9,-2) rectangle (12,-3);
 \draw[pattern=crosshatch, pattern color= gray] (12,-2) rectangle (18,-3);
\draw(18,0) rectangle (28,-3); \draw(0,-3) rectangle (18,-21);
  \filldraw[draw=black,color=black!50!white] (0,-3) rectangle (3,-6);
  \filldraw[draw=black,color=black!50!white] (3,-6) rectangle (6,-9);
  \filldraw[draw=black,color=black!50!white] (6,-9) rectangle (9,-12);
  \filldraw[draw=black,color=black!50!white] (9,-12) rectangle (12,-15);
  \filldraw[draw=black,color=black!50!white] (12,-15) rectangle (15,-18);
  \filldraw[draw=black,color=black!50!white] (15,-18) rectangle (18,-21);
\draw (18,-3) rectangle (28,-21);
 \draw[pattern=dots, pattern color=black] (18,-3) rectangle (21,-6);
 \draw[pattern=dots, pattern color=black] (19,-6) rectangle (20,-9);
 \draw (20,-6) rectangle (21,-9);
 \draw[pattern=dots, pattern color=black] (21,-6) rectangle (23,-9);
 \draw[pattern=dots, pattern color=black] (20,-9) rectangle (21,-12);
 \draw (21,-9) rectangle (22,-12);
 \draw[pattern=dots, pattern color=black] (22,-9) rectangle (24,-12);
 \draw[pattern=dots, pattern color=black] (21,-12) rectangle (22,-15);
 \draw (22,-12) rectangle (24,-15);
 \draw[pattern=dots, pattern color=black] (24,-12) rectangle (26,-15); 
 \draw[pattern=dots, pattern color=black] (22,-15) rectangle (23,-18);
 \draw (23,-15) rectangle (25,-18);
 \draw[pattern=dots, pattern color=black] (25,-15) rectangle (27,-18);
 \draw[pattern=dots, pattern color=black] (23,-18) rectangle (24,-21);
 \draw (24,-18) rectangle (26,-21);
 \draw[pattern=dots, pattern color=black] (26,-18) rectangle (28,-21);
\end{tikzpicture}.
\end{align*}
\end{example}

Observe that for codepth 3, the differential maps $\partial^F_i$ for $i\geq 3$ also can be built from two sets of matrices: $\{\partial_1,\partial_2,\partial_3\}$ which is independent of the ideal $I$, and $\{\zeta_1^0, \zeta_2^0,\zeta_3^0\}$. The number of blocks and the spreading are different from the above  Example~\ref{exp:codepth2 edim3} for codepth 2, as follows:
\begin{itemize}
    \item[--] When $i=2k+3$, we need $\partial_1$, $\partial_3$, and $\zeta_2^0$ as building blocks: $\partial_1$ appears diagonally $k+3\choose 2$ times, $\partial_3$ appears diagonally $k+2\choose 2$ times, and $\zeta_2^k$ is built from $\zeta_2^0$ $k+2\choose 2$ times by spreading it, as described above.
    \item[--] When $i=2k+4$, we need $\partial_2$, $\zeta_1^0$, and $\zeta_3^0$ as building blocks, with $\partial_2$ appears diagonally $k+3\choose 2$ times, $\zeta_1^k$ is built from $\zeta_1^0$ $k+3\choose 2$ times by spreading it, and $\zeta_3^k$ is built from $\zeta_3^0$ $k+2\choose 2$ times by spreading it, as described above. 
\end{itemize}

\begin{remark} As in Examples \ref{exp:codepth2 edim3} and \ref{exp:codepth3 edim3} one can obtain $\zeta_u^k$ from $\zeta_u^0$ just by a splitting-and-spreading process. Therefore, by Theorem~\ref{ci codepth c}, the differential maps $\partial^F$ in the minimal free resolution of $\sk$ over $R$ can be obtained explicitly by using only $2n$ matrices, where $n$ is the embedding dimension of $R$:
\[
\partial_1^K,\dots, \partial_{n}^K \qquad \text{and} \qquad
\zeta^0_1,\dots,\zeta^0_{n}.
\]
\end{remark}

\appendix
\section{}
\label{appendix}


Here we describe the connection between our work and the Tate resolution given in \cite{T}. As before, let $(R,\fm,\sk)$ be a complete intersection ring of codepth $c \geq 1$ and $(K,\partial)$ be the Koszul complex on a minimal set of generators of $\fm$. Let $V$ be a free $R$-module of rank $c$ with a basis $\{v_i\}_{1 \leq i \leq c}$ and $D^\bullet(V)$ denote the divided power functor on $V$. Recall that $D^\bullet(V)$ is a module over the symmetric algebra $S_\bullet (V^*)$ of the dual module $V^*\colon=\Hom{R}{V}{R}$ of $V$. Let $z_1, \ldots, z_c \in K_1$ be representatives of a basis of $\HH{1}{K}$.  By adjoining variables, the acyclic closure of the residue field $\sk$ over $R$
 as an algebra may be described as the tensor product $K \otimes_R D^\bullet(V)$ with differential maps given by 
 \[z\otimes w\mapsto \partial(z)\otimes w + \sum_{i=1}^c (z_i\wedge z)\otimes (v_i^*\cdot w),\]
where $z \in K$, $w \in D^\bullet(V)$, and $v_i^*$ denotes a basis element of $V^*$ dual to $v_i$. 

With this setting, a weakly increasing tuple $(u_1,\dots,u_k)$, with $1 \leq u_1 \leq\cdots \leq u_k \leq c$, can be put in a bijection with a basis element of $D^k(V)$ as follows:
\[(u_1,u_2,\dots,u_k)=(\underbrace{i_1\dots i_1}_{j_1 \text{ times }}\underbrace{i_2\dots i_2}_{j_2 \text{ times }}\dots \underbrace{i_\ell\dots i_\ell}_{j_\ell \text{ times }}) \longleftrightarrow v_{i_1}^{(j_1)}v_{i_2}^{(j_2)}\dots v_{i_\ell}^{(j_\ell)},\]
where $1\leq i_1<\dots<i_\ell\leq c$ and $k=j_1+\dots+j_\ell.$
For example, $1112334 \leftrightarrow v_1^{(3)}v_2v_3^{(2)}v_4$. Thus, using the above correspondence, there is a bijection between $K^{\oplus{k+c-1 \choose c-1}}$ and $K \otimes_R D^k(V)$ given by
\[ \pi=(\pi_{u_1\dots u_k})_{1\leq u_1\leq\dots\leq u_k\leq c}\mapsto \sum_{1\leq u_1\leq\dots\leq u_k\leq c}\pi_{u_1\dots u_k}\otimes  v_{i_1}^{(j_1)}v_{i_2}^{(j_2)}\dots v_{i_\ell}^{(j_\ell)}.\] 

Our Lemma \ref{zeta chain c} can be reformulated as follows.
\begin{lemma} 
\label{mu chain c}
Let $\mu^k \colon K \otimes_R D^{k+1}(V) \to K \otimes_R D^{k}(V)$ denote the map given by left multiplication by the element $\sum_{i=1}^c z_i \otimes v_i^*$. 
Then, for all $k\geq 0$, $\mu^k$ is a chain map and $\mu^k \circ \mu^{k+1}=0$.
\end{lemma}

Under this translation, Proposition \ref{prop:zeta c} becomes the following result. For completeness, we sketch some ideas for a different proof using the language of divided powers.
\begin{proposition}
Let $(R,\fm,\sk)$ be a complete intersection local ring of codepth $c \geq 1$. Let $A=A_0\oplus A_1\oplus \cdots \oplus A_c$ be the Koszul homology and $V$ be a free $R$-module of rank $c$. Then for each integer $k\geq 0$, the following sequence is exact
\begin{equation}
\label{ses mu codepth c}
    0\to A_0\otimes_R D^{k+1}(V)\xrightarrow{[\mu^{k}_1]}A_1 \otimes_R D^{k}(V) \xrightarrow{[\mu^{k-1}_2]}A_2 \otimes_R D^{k-1}(V) \to \cdots \xrightarrow{[\mu^{k-c+1}_{c}]}A_c \otimes_R D^{k+1-c}(V)\to  0,
\end{equation}
where the maps $[\mu^{k}]$ are given by left multiplication by $\sum_{i=1}^c [z_i] \otimes v_i^*$. 
\end{proposition}

\begin{proof}[Idea of proof]
Since $R$ is a complete intersection ring, $A \cong \bigwedge^\bullet A_1$ and there is an isomorphism $A_1 \cong \sk \otimes_R V$ via the identification $[z_i] \leftrightarrow 1\otimes v_i$. Thus, the complex \eqref{ses mu codepth c} is isomorphic to $\sk \otimes_R -$ applied to the complex of $R$-modules:
\[0 \to D^{k+1}(V) \to V \otimes_R D^k(V) \to \cdots \to \bigwedge^c V \otimes_R D^{k+1-c}(V) \to 0. \]
Taking the $R$-dual, this is equivalently the degree $k+1$ strand (c.f. \cite[A2.6.1]{Ei}) of the Koszul complex over the symmetric $R$-algebra $S_\bullet (V^*)$, which is exact since the Koszul complex over $S_\bullet (V^*)$ is a free resolution of $R$. 
\end{proof}

One could continue with this translation to get an alternate proof that $K \otimes_R D^\bullet(V)$ is a resolution of the residue field $\sk$.

\section*{Acknowledgements} The authors thank the referee for useful suggestions to improve the paper. Nguyen was partially supported by the Naval Academy Research Council and NSF grant DMS-2201146.


\end{document}